\documentclass[a4,12pt]{article}

\usepackage{latexsym,amsthm,amsmath,amscd,amssymb,framed,graphics,mathrsfs}
\usepackage{enumerate}
\usepackage{bbm}

\usepackage{graphicx}
\usepackage[utf8]{inputenc}

\usepackage{natbib}
\renewenvironment{thebibliography}[1]{%
    \begin{oldthebibliography}{#1}%
    \setlength{\parskip}{0ex}%
    \setlength{\itemsep}{0ex}%
}%
                 {%
  \end{oldthebibliography}%
                 }

\usepackage[colorlinks = true,
            linkcolor = blue,
            urlcolor  = blue,
            citecolor = blue,
            anchorcolor = blue]{hyperref}
\usepackage{url}
\usepackage{colordvi}
\usepackage{color}
\usepackage{bbm}
\usepackage[super]{nth}
\usepackage{overpic}
\usepackage[all]{xy}
\usepackage{soul}
\setstcolor{red}

\makeatletter

\@addtoreset{figure}{section}
\def\thefigure{\thesection.\@arabic\c@figure}
\def\fps@figure{h, t}
\@addtoreset{table}{bsection}
\def\thetable{\thesection.\@arabic\c@table}
\def\fps@table{h, t}
\@addtoreset{equation}{section}

\makeatother

\newtheorem{theorem}{Theorem}[section]
 
\newtheorem{proposition}[theorem]{Proposition}

\theoremstyle{definition}

\newtheorem{remark}[theorem]{Remark}

\newcommand{\pp}[2]{\frac{\partial #1}{\partial #2}} 

\DeclareMathOperator{\diff}{d}

\begin{document}
\title{Higher order phase averaging for highly oscillatory systems}

\author{Werner Bauer, Colin J. Cotter, Beth Wingate}

\maketitle

\newcommand{\V}{\mathbf{V}}
\newcommand{\U}{\mathbf{U}}
\newcommand{\W}{\mathbf{W}}
\newcommand{\M}{\mathrm{M}}
\newcommand{\R}{\mathrm{R}}
\newcommand{\La}{\mathrm{L}}
\newcommand{\Fu}{\mathbf{F}}

\begin{abstract}

  We introduce a higher order phase averaging method for nonlinear
  oscillatory systems. Phase averaging is a technique to filter fast
  motions from the dynamics whilst still accounting for their effect
  on the slow dynamics. Phase averaging is useful for deriving reduced
  models that can be solved numerically with more efficiency, since
  larger timesteps can be taken. Recently, Haut and Wingate (2014)
  introduced the idea of computing finite window numerical phase
  averages in parallel as the basis for a coarse propagator for a
  parallel-in-time algorithm. In this contribution, we provide a
  framework for higher order phase averages that aims to better
  approximate the unaveraged system whilst still filtering fast
  motions. Whilst the basic phase average assumes that the solution is
  independent of changes of phase, the higher order method expands the
  phase dependency in a basis which the equations are projected onto.
  In this new framework, the original numerical phase averaging
    formulation arises as the lowest order version of this expansion in
    which the nonlinearity is projected onto the space of functions that
    are independent of the phase. Our new projection onto functions that
    are $k$th degree polynomials in the phase gives rise to higher
    order corrections to the phase averaging formulation. 
  We illustrate the properties of this method on an ODE that describes
  the dynamics of a swinging spring due to Lynch (2002). Although
  idealized, this model shows an interesting analogy to geophysical
  flows as it exhibits a slow dynamics that arises through the
  resonance between fast oscillations. On this example, we show
  convergence to the non-averaged (exact) solution with increasing
  approximation order also for finite averaging windows. At zeroth
  order, our method coincides with a standard phase average, but at
  higher order it is more accurate in the sense that solutions of the
  phase averaged model track the solutions of the unaveraged equations
  more accurately.\\
  

  \noindent\textbf{MSC\,2020:} 37M99
 
\end{abstract}

\section{Introduction}
\label{sec:intro}
          
          One of the main challenges in simulating highly oscillatory multiscale PDEs is that the fast waves are known to create the lowest frequencies of the solution through nonlinear phase shifts, but are also responsible for severe linear wave time-step restrictions.
          
Phase averaging is a technique for approximating highly oscillatory
systems so that the solution still captures the long time trends in
the dynamics \citet{SandersVerhulst}.
In the context of highly oscillatory nonlinear PDEs,
the combination of coordinate transformation, phase averaging, and asymptotics has been used to
derive slow PDEs that do this \citep[for
  example]{schochet1994fast,majda1998averaging,Babin1997,wingate2011low},
providing rigorous derivations of slow geophysical models such as the
quasigeostrophic model. Slow models are more analytically tractable
and easier to explain, and they are more efficient to solve
numerically. For an example numerical method inspired from the study of fast singular limits see \citep{Jones1999}.  The reason they are more tractable to solve is because they have had the fast frequencies
filtered from them, and therefore much larger timesteps can be taken without
excessively compromising stability or accuracy. For example, the first
numerical weather prediction exercise by Charney, Fj\/ortoft and von
Neumann \citep{ChFjNe1950} was based upon the barotropic vorticity equation, a model that
does not support the full gravity and pressure waves present in more
detailed models. One problem with asymptotic models is that they are only accurate at or near the limit, whereas the physically realistic parameters used in applications may not be near these limits, or the parameter regime itself may shift as the simulation unfolds. Ongoing research in theoretical fluid dynamics highlights an issue with using asymptotic equations that remove the fast oscillations: waves are responsible for transfering the energy toward the low frequencies, see \citep{SmithWaleffe2002}.  This is one reason asymptotic models are not used in contemporary weather and climate modeling as a predictive tool. In fact, the role of the oscillations has been found to be key to allowing accurate simulations and in some operational weather forcasting models, the compressible Euler equations, which include sound waves, are used; for many years sound waves were thought to be unimportant \citep{Davies_etal_03}.

This paper is not about numerical algorithms, but it is highly
motivated by them. \citet{haut2014asymptotic} investigated the
possibility of using a framework inspired by the study of fast
singular limits in that they use a coordinate transformation and phase
averaging, but without asymptotics. Their goal was to use
  phase averaging to reduce high oscillations in the transformed
  solution, so that larger timesteps can be taken without triggering
  instability or producing large errors.  They performed phase
averaging over a finite window, and the phase average integral with
numerical quadrature. Although this results in a computationally more
intensive model, the terms in the quadrature sum can be evaluated
independently, exposing possibilities for parallel computing. The
phase averaged formulation also involves the application of matrix
exponentials, which can be computed by Krylov methods or other
techniques such as the parallel rational approximation approach of
\citet{haut2016high,CALIARI2021110289,schreiber2018beyond},
particularly suited to oscillatory operators. This produces a variant
of heterogeneous multiscale methods
\citep{abdulle2012heterogeneous,Kevrekidis2003,tao2010}. The accuracy
and stability of the finite window phase averaging technique was
studied in \citet{peddle2019parareal}, and applied to proving
convergence of the parareal iteration.  If the results of the
averaging method are insufficiently accurate, then it can be used as
the predictor in a predictor-corrector framework, see \citep[for
  example]{Jones1999}.  Many of these frameworks provide additional
opportunities for parallel computing. With this in mind,
\citet{haut2014asymptotic} introduced an asymptotic parareal method
using the finite window averaging as a coarse propagator, and using a
standard time integration method applied to the original unaveraged
system as a fine propagator. This approach was extended to
  ODEs with nonlinear fast oscillations in \citet{ariel2016}.

In this paper we lay the foundations for an alternative approach to
correcting averaging errors, where the accuracy of the finite window
averaged model is increased ``from the top down'' by extending the
phase average into a polynomial expansion of the solution in the phase
variable. The goal is to further reduce the error in the slow
component of the solution whilst still filtering fast
oscillations. This increases the model complexity but it may be
interesting to consider how to hide higher order corrections to the
solution behind other computations in parallel, as occurs in the
PFASST framework \citep{minion2011hybrid}, for example.  Here, we
introduce and explore the higher order averaging technique applied to
ODEs, although we have highly oscillatory PDEs and fast singular
limits in mind when we do this. To isolate and understand the effects
of our higher order averaging technique, in this paper we compute
phase average integrals exactly (even though practical approaches
would use numerical quadrature), and we use an adaptive timestepper to
get a very accurate estimate of the ODE solution trajectory. We find
that when applied to a swinging spring model introduced by Lynch
\citep{holm2002stepwise}, the higher order averaging does indeed
produce a more accurate solution than the basic finite window
averaging, and that this error decreases as we increase the polynomial
degree of the expansion in the phase variable.

The rest of the paper is organised as follows. In Section
\ref{sec:expint} we introduce the higher order averaging technique in
the most general form, before specialising to a polynomial basis for
the expansion in the phase variable as well as selecting a Gaussian
averaging weight function so that the phase averaging can be computed
analytically. This is not critical to the framework but was done to
expediate our numerical experiments and to expose understanding of the
effect of the higher order averaging.  We then discuss asymptotic
limits for the averaging window. In Section \ref{sec:swingspring} we
develop further formulae for Lynch's swinging spring model, a specific
low dimensional example.  In Section \ref{sec:num} we present
numerical integrations for the swinging spring example, calculating
errors between the higher order average models and the exact model.
Finally, in Section \ref{sec:summary} we provide a summary and
outlook.

\section{Higher order averaging: formulation}
\label{sec:expint}

We consider initial value problems written in the following abstract
form
\begin{equation}\label{equ_gen_U}
 \pp{}{t} \U = L\U + N(\U ,\U ),
\end{equation}
where $L$ is a linear operator (with 
purely imaginary eigenvalues) and 
$N$ is a nonlinear operator. 
Restricting ourselves to quadratic nonlinearities, we assume
$N$ to be a bilinear form (bilinearity indicated by the two slots).
To develop phase averaged models, it is useful
to reformulate \eqref{equ_gen_U}
by introducing the factorization
\begin{equation}\label{factoring}
 \V(t) = e^{-t L} \U(t),
\end{equation}
where
the \emph{modulated variable} $\V(t)$ satisfies the equation
\begin{equation}\label{equ_gen_V}
\pp{}{t} \V(t) = e^{-t L} N \left(e^{ t L} \V(t),e^{ t L} \V(t)\right).  
\end{equation}
We refer to this form as the \emph{modulated equation}.  It is
straightforward to show that both formulations are equivalent:
when multiplying both sides of \eqref{equ_gen_U} with the integration
factor $e^{-t L} $ and noting that $e^{-t L} (\partial_t \U - L \U) =
\partial_t (e^{-t L} \U )$, we recover Eqn.~\eqref{equ_gen_V} via
 \begin{align}
 e^{-t L} (\partial_t \U - t L \U) =  e^{-t  L} N (\U,\U) 
 &\Leftrightarrow \partial_t (e^{-t L} \U ) =  e^{-t L} N (\U,\U) \\
 &\Leftrightarrow \partial_t \V  =  e^{-t L} N (e^{t L} \V,e^{t L}\V) .
 \end{align}
 Note that this equation is valid for cases with 
 or without scale separation (as discussed below).

To apply phase averaging, we add an additional phase parameter $s \in
\mathbbm{R}$ over which we integrate and hence average the solution.
As a first step, we assume a factorization
 \begin{equation}\label{factoring_s}
 \V(t,s) = e^{-(t+s) L} \U(t)
\end{equation}
where, besides the time parameter $t$, we include a \emph{phase
  parameter} $s\in \mathbbm{R}$ in the exponent such that the solution
in $\V$ depends on both $t$ and $s$. The phase parameter $s$
corresponds to a change in the point where the exponential
$\exp((t+s)L)$ is equal to the identity, i.e. a phase shift.

As a calculation similar to above shows, this factorization
yields a form of \eqref{equ_gen_V} that includes the parameter $s$,
i.e.
\begin{equation}\label{equ_gen_V_s}
   \pp{}{t} \V(t,s) = e^{-(t+s) L} N \left(e^{ (t+s) L} \V(t,s),e^{ (t+s) L} \V(t,s)\right).
\end{equation}
 This gives us a one parameter family of solutions $\V(t,s)$ that
 depend on the phase parameter $s$ with the property that for $s = 0$
 we recover the solution of \eqref{equ_gen_V}. We can then obtain
 a phase averaged model,
  \begin{equation}
   \pp{}{t} \bar{\V}(t) = \lim_{T\to \infty}
   \frac{1}{2T}\int_{-T}^{T} e^{-(t+s) L} N \left(e^{ (t+s) L} \bar{\V}(t,s),e^{ (t+s) L} \bar{\V}(t,s)\right) \diff s.
\end{equation}
 Phase averaging for \eqref{equ_gen_V_s} was analysed by
 \citet{schochet1994fast} in the limit $\epsilon\to 0$, leading to
 slow equations that approximate the solution of the unaveraged model
 in this limit in some appropriate sense. In particular,
 \citet{majda1998averaging} showed that this averaging and
 $\epsilon\to 0$ limit recovers the quasigeostrophic equation in the
 low Rossby number limit of the rotating shallow water equations, when
 written in this form. For finite Rossby number,
 \citet{haut2014asymptotic} proposed to adapt this approximation to
 numerical computations by keeping $T$ finite and introducing a
 $C^\infty$ weight function $\rho$ with $\rho\geq 0$,
 $\|\rho\|_{L^1}=1$, to obtain the averaged equation
 \begin{equation}
   \label{equ_gen_V_int}
   \pp{}{t} \bar{\V}(t) = \int_{-\infty}^{\infty}
   \rho_T\left(s\right)e^{-(t+s) L} N \left(e^{ (t+s) L}
   \bar{\V}(t,s),e^{ (t+s) L} \bar{\V}(t,s)\right) \diff s,
 \end{equation}
 where $\rho_T(s)=\rho(s/T)/T$. They further proposed to replace the
 integral via quadrature (choosing $\rho$ to have compact support),
 with the evaluation of the integrand at each quadrature point being
 computed independently in parallel; this formulation provided the
 coarse propagator for a parallel-in-time algorithm. The
   reason for the phase averaging in that paper, and this one, is to
   remove high oscillations from the modulated variable, allowing
   larger timesteps to be taken in a numerical solver. This averaging
   introduces an error in the solution, and the goal of this paper is
   to introduce refinements to the phase averaging procedure so that
   this error can be reduced.

Note that depending on the value of $T$,
formulation~\eqref{equ_gen_V_int} recovers an averaged version, an
asymptotic approximation, or the full equations~\eqref{equ_gen_V}:
\begin{itemize}
\item for $T \rightarrow 0$, we recover Eqn.~\eqref{equ_gen_V}
  because in the limit $T\rightarrow 0$ the weight function
  $\rho$ converges formally to the Dirac measure $\delta(s)$.
\item for 
  $T \rightarrow \infty$, we recover \eqref{equ_gen_V_int},
 because when integrating the resulting integral value is invariant under 
 shifts in $t$; hence for $t=0$ we recover the asymptotic approximation introduced in 
 \citet{majda1998averaging} after also taking $\epsilon \to 0$.
 
\item for $0<< T << \infty$, formulation~\eqref{equ_gen_V_int} 
  has a smoothing effect on fast oscillations with time period less than $T$.
\end{itemize}
The optimum choice of $T$ as a function of $\epsilon$ to minimise the
error due to averaging when combined with a numerical time integration
method to solve \eqref{equ_gen_V_int} was investigated in
\citet{peddle2019parareal}.

An interpretation of \eqref{equ_gen_V_int} is that we have projected
$\V(t,s)$ onto the space of functions independent of $s$ using the
$\rho$-weighted $L^2$ inner product. In the next subsection, our
higher order averaging method for finite averaging window $T$ will be
based on projecting the one parameter family of solutions on to
higher order polynomials in $s$. After solving this system of coupled
differential equations, we set $s = 0$ because we are interested only
in the time evolution of the solution with this phase value. Hence, we
only present solutions at $s=0$ in the numerical results section.

\subsection{General framework of higher order time averaging}

We abbreviate the RHS of \eqref{equ_gen_V} as general bilinear function 
$\Fu(\V(t,s),\V(t,s);t+s)$ that has quadratic dependency on $\V(t,s)$, i.e.
\begin{equation}\label{eqn_general}
  \pp{}{t}  \V(t,s)  = \Fu\big( \V(t,s),\V(t,s);t+s \big) .
\end{equation}
According to \eqref{factoring_s}, the solution vector $\V$ depends on both $t$ and $s$ and 
we make the approximation that it can be represented by the factorization 
\begin{equation}\label{eqn_vel_split}
  \V(t,s) = \sum_{k = 0}^p \V_k(t) \phi_k(s),
\end{equation}
where $\phi_k(s)$ are basis vectors depending on $s$ that span $Q$,
some chosen $(p+1)$-dimensional function space.

We project \eqref{eqn_general} onto $Q$ to obtain
\begin{equation}
 \int_{-\infty}^{\infty} \rho_T\left({s}\right) \W(s) \cdot \pp{}{t}  \V(t,s) ds 
 =  \int_{-\infty}^{\infty} \rho_T\left({s}\right) \W(s) \cdot \Fu \big( \V(t,s),\V(t,s); t+s \big) ds \quad \forall \W(s)\in Q,
 \end{equation}
 where $\rho_T(s)=\rho(s/T)/T$ for our chosen weight
 function $\rho$.

Also expanding the test function $\W$ in the basis,
\begin{equation}
 \W(s)  = \sum_{j = 0}^p  \W_j  \phi_j(s),
\end{equation}
we get
\begin{equation}
  \begin{split}
    \sum_{j=0}^p \sum_{k = 0}^p &
  \W_j\int_{-\infty}^{\infty} \frac{1}{T}\rho\left(\frac{s}{T}\right)   {\phi_j(s) \phi_k (s) ds}\pp{}{t}  \V_k(t)  =  \\
  &\sum_{j=0}^p \W_j \int_{-\infty}^{\infty} \frac{1}{T}\rho\left(\frac{s}{T}\right) \phi_j(s)  \Fu \big( \sum_{k = 0}^p  \V_k(t)  \phi_k(s),\sum_{l = 0}^p  \V_l(t)  \phi_l(s); t+s \big) ds, \quad \forall \,\W_1,\ldots,\W_p.
  \end{split}
 \end{equation}
 As the coefficients $\W_j$ are independent from each other, we obtain
 \begin{equation}
  \begin{split}
    \sum_{k = 0}^p 
   \int_{-\infty}^{\infty} \frac{1}{T}\rho\left(\frac{s}{T}\right)  & {\phi_j(s) \phi_k (s) ds}\pp{}{t}  \V_k(t)  =  \\
  &    \int_{-\infty}^{\infty} \frac{1}{T}\rho\left(\frac{s}{T}\right) \phi_j(s)  \Fu \big( \sum_{k = 0}^p  \V_k(t)  \phi_k(s),\sum_{l = 0}^p  \V_l(t)  \phi_l(s); t+s \big) ds.
  \end{split}
 \end{equation}
 This defines our averaged model for the components $\V_k$, in the
 form of a matrix-vector equation for $\pp{}{t}\V_k$.  To use it, we
 set $\V(t=0,s)=\U (0)$, the initial condition, independently of $s$. 
 Then we solve the equations forward in
 time (perhaps approximating the phase integral as a sum), 
 taking the solutions at $s=0$.

 To understand more about what this approximation does, we further
 specify the nonlinear function $\Fu$.  As it is bilinear, we may
 pull out the double sum over the basis $\phi$, i.e.
 \begin{equation}
 \begin{split}
  \Fu \left( \sum_{k = 0}^p  \V_k(t)  \phi_k(s),\sum_{l = 0}^p  \V_l(t)  \phi_l(s); t+s \right) = 
  & \sum_{k = 0}^p  \sum_{l = 0}^p \phi_k(s) \phi_l(s) \Fu \big(  \V_k(t)  , \V_l(t)  ; t+s \big). 
  \end{split}
  \end{equation}
  Moreover, we assume that we can factor out the time dependency of $\Fu$ on $t+s$ 
  so that $\Fu$ is represented as
 \begin{equation}
 \Fu \big(  \V_k(t)  , \V_l(t)  ; t+s \big) = \sum_{m  } \Fu_m \big(  \V_k(t)  , \V_l(t) \big)e^{i {(t+s)c_m} }
\end{equation}
where the sum in $m$ is over different terms that are quadratic in
$\V$ (for different polynomial order $k,l$) with coefficients $c_m$
and bilinear functions $\Fu_m$ and that depend on the nonlinear operator $N$. 
In the case of a PDE, this sum over $m$ will be replaced by
an infinite sum over frequencies.
Examples of these coefficients and their relations are shown further
below in Table~\ref{table_coeff} for the swinging spring model.  To
ease notation, we will adopt the short hand notation
$\Fu_{m,k,l}(t):= \Fu_m \big( \V_k(t) , \V_l(t) \big)$.

 In summary, we end up with a system of equations (for $j=0,\dots p$)
 that reads
 \begin{equation}\label{eqn_sum}
  \begin{split}
    \sum_{k = 0}^p 
   &\int_{-\infty}^{\infty} \frac{1}{T}\rho\left(\frac{s}{T}\right)   {\phi_j(s) \phi_k (s) ds}\pp{}{t}  \V_k(t)  =  \\
  &      \sum_{m} \sum_{k = 0}^p \sum_{l = 0}^p \Fu_{m,k,l}(t)  
  \int_{-\infty}^{\infty}\frac{1}{T} \rho\left(\frac{s}{T}\right) \phi_j(s)  
    \phi_k(s) \phi_l(s)e^{i {(t+s)c_m} } 
 ds \quad \forall  \phi_j(s).
  \end{split}
 \end{equation}
 This formulation provides a general framework of higher order
 averages for various choices of $Q$ and weight function $\rho$.
 
 \subsection{Polynomial higher order time averaging} 

 In this section we make some specific calculations using the choice
 of $Q=P_p$, the degree-p polynomials, and a particular choice of
 weight function that allows us to make progress analytically, 
 namely $\rho (s) =\exp(-s^2/2)/\sqrt{2\pi}$, recalling that
 \begin{equation}\label{I_norm}
 \int_{-\infty}^{\infty}\frac{1}{T} \rho\left(\frac{s}{T}\right) ds=
 \int_{-\infty}^\infty \frac{1}{\sqrt{2\pi T^2}} e^{-\frac{s^2}{2T^2}} ds = 1.
 \end{equation}
 This allows us to integrate over the whole $\mathbbm R$ while guaranteeing a bounded integral,
 but other choices are possible too.

 As a basis $\phi$, we choose the standard monomials
 \begin{equation}
  \phi_j(s) = s^j, \quad j = 0, \dots, p.
 \end{equation}
 This is a poorly-conditioned basis, and better choices might be made
 for large values of $p$ (for the weight function we are using the
 Hermite polynomials would be ideal since they diagonalise the mass
 matrix). However, the monomial basis simplifies the calculations
 below. Also, under this basis, ${\V}(t,s=0)= {\V}_0(t)$, so to initialise
 the system we can just set ${\V}_0(0)=\V(0,s=0)$ and ${\V}_k(0)=0$ for
 $1\leq k \leq p$.
 
 Substituting this basis into \eqref{eqn_sum}, we arrive at
 \begin{equation}\label{eqn_sum1}
  \begin{split}
    \sum_{k = 0}^p 
   & \frac{1}{\sqrt{2\pi T^2}}\int_{-\infty}^\infty e^{-\frac{s^2}{2T^2}}    {s^j s^k  ds } \pp{}{t}  \V_k(t)  =  \\
  &      \sum_{m} \sum_{k = 0}^p \sum_{l = 0}^p\Fu_{m,k,l}(t)  e^{ic_m t}
  \frac{1}{\sqrt{2\pi T^2}}\int_{-\infty}^\infty e^{-\frac{s^2}{2T^2}}  e^{i {c_m} s}  
  s^j s^k s^l 
 ds \quad \forall  j.
  \end{split}
 \end{equation}
 The latter equation can be formulated more concisely as
 \begin{equation}\label{eqn_sum2}
  \begin{split}
    \sum_{k = 0}^p \M_{jk}  \pp{}{t}  \V_k(t)  =  
    \sum_{m } \sum_{k = 0}^p \sum_{l = 0}^p\Fu_{m,k,l}(t)  e^{i {c_m} t}
    \tilde\R_{jkl}^m \quad \forall  j,
  \end{split}
 \end{equation}
 in which we denote the integral on the left hand side (LHS) of \eqref{eqn_sum1} as a mass matrix 
 \begin{equation}\label{def_massmtx}
  \M_{jk}  = \frac{1}{\sqrt{2\pi T^2}}\int_{-\infty}^\infty e^{-\frac{s^2}{2T^2}}    {s^j s^k  ds },  
 \end{equation}
 and that on the RHS of \eqref{eqn_sum1} as
 \begin{equation}\label{def_RHSmassmtx}
 \tilde \R_{jkl}^m  =  \frac{1}{\sqrt{2\pi T^2}}\int_{-\infty}^\infty e^{-\frac{s^2}{2T^2}} e^{i {c_m} s}   s^j s^k s^l  ds .
 \end{equation}
 To evaluate these integrals, it will be useful to introduce the notation
 \begin{equation}\label{equ_Ialpha}
  I_\alpha := \frac{1}{\sqrt{2\pi T^2}}\int_{-\infty}^\infty e^{-\frac{s^2}{2T^2}}    {s^\alpha  ds }
 \end{equation}
 with $I_0 = 1$ which follows immediately from \eqref{I_norm}. 
 Then, for instance, for $\alpha = i +j$ we have $M_{ij} = I_\alpha$. 
 With the help of this concise notation, we formulate the following two propositions that 
 allow us to evaluate the integrals in \eqref{def_massmtx} and \eqref{def_RHSmassmtx}.

 \begin{proposition}\label{prop_I}
  For averaging window $T$ and $\alpha = j + k$ with $j,k = 0, \dots, p$,
  the integral $I_\alpha$ in \eqref{equ_Ialpha}   can be represented in 
  the recursive form
 \begin{equation}\label{equ_I}
  I_\alpha = T^2 (\alpha-1) I_{\alpha-2} \quad \text{with} \quad I_0 = 1  
 \end{equation}  
 In particular, $I_{\alpha} = 0$ for all $\alpha$ odd.
 \end{proposition}

 \begin{proof} This recursive formula follows directly from integration by parts. Using the normalization
 factor $\gamma = \frac{1}{\sqrt{2\pi T^2}}$, there follows
  \begin{equation}\label{eqn_intpartsaa}
  \begin{split}
   I_\alpha = \gamma \int_{-\infty }^{\infty  } 
  \frac{d}{ds} e^{- \frac{s^2}{2T^2} } \frac{-2T^2}{2s} s^\alpha ds  
 & = \gamma \left[- e^{- \frac{s^2}{2T^2} } T^2 s^{\alpha-1} \right]_{-\infty }^{\infty  } 
  + T^2 (\alpha-1) \gamma \int_{-\infty }^{\infty  } e^{- \frac{s^2}{2T^2} } s^{\alpha-2}  ds \\
 &  = T^2 (\alpha-1) I_{\alpha-2}.
  \end{split}
 \end{equation}
 For $\alpha = 0$, we have $I_0 = 1$ since $\int_{-\infty }^{\infty  } e^{- \frac{s^2}{2T^2} } = \sqrt{2\pi T^2}$
 which follows from \eqref{I_norm}. 
 For odd $\alpha$, $I_\alpha$ is an integral of a product of an even and odd function over entire $\mathbbm{R}$, 
 hence zero.
 \end{proof}

 \begin{proposition}
 \label{prop_R}
 Using the recursive formula of Proposition~\ref{prop_I} 
 and for averaging window $T$ and $\alpha = j + k +l$ with $j,k,l = 0, \dots, p$,
  the integral~\eqref{def_RHSmassmtx} can be represented for all $m$ as
 \begin{equation}
   \tilde \R_{\alpha}^m  = e^{- \frac{c_m^2 T^2}{2} }   \R_\alpha^m 
\quad \text{with} \quad 
  \R^m_\alpha =  \sum_{\beta = 0}^\alpha \binom{\alpha}{\beta} 
  I_{\alpha - \beta} \cdot (ic_m T^2)^{\beta} 
 \end{equation}
 with the property that for $\alpha =0$ we have $\R^m_0=1$ for all $m$.
 \end{proposition}
  
 \begin{proof}
 First, we rewrite \eqref{def_RHSmassmtx} using $\alpha = j+k+l$ and by 
 completing the square 
 ($ -\frac{s^2}{2T^2} + ic_m s =     - \frac{1}{2T^2} (s - ic_m T^2)^2 - \frac{c_m^2 T^2}{2}$) 
 we arrive at
 \begin{equation}\label{eqn_mod_completion}
 \tilde \R_{\alpha}^m  = e^{- \frac{c_m^2 T^2}{2} }   \R_\alpha^m  \quad 
  \text{with} \quad  \R_\alpha^m:= \frac{1}{\sqrt{2\pi T^2}} \int_{-\infty}^\infty e^{- \frac{1}{2T^2} (s - ic_m T^2)^2}  s^\alpha ds .
 \end{equation}
 Next, we use the substitution $s' = s - ic_m T^2$ with $ds' = ds$:
 \begin{equation}
  \R_\alpha^m =   \frac{1}{\sqrt{2\pi T^2}} \int_{-\infty-  ic_m T^2 }^{\infty-  ic_m T^2 } e^{- \frac{s'^2}{2T^2} }  (s' + ic_m T^2)^\alpha ds' .
 \end{equation}
 Because the latter is a contour integral in the complex domain in which 
 one term is integrated along the real axis and another one along a 
 semi circle at infinity, the shift of the real axis in the complex direction by $-  ic_m T^2$ 
 does not change the integral value as long as there are no, or the same number of singularities 
 enclosed by the contour. 
 Therefore, we can equivalently write
 \begin{equation}\label{equ_ralpha}
  \R_\alpha^m =   \frac{1}{\sqrt{2\pi T^2}} \int_{-\infty}^{\infty} e^{- \frac{s'^2}{2T^2}}  (s' + ic_m T^2)^\alpha ds'
 \end{equation}
 while for $\alpha = 0$ there follows
 \begin{equation}\label{eqn_r3a}
  \R_0^m  =  {\frac{1}{\sqrt{2\pi T^2}} \int_{-\infty - ic_m T^2}^{\infty - ic_m T^2} e^{- \frac{s'^2}{2T^2} }   ds'}
  =  {\frac{1}{\sqrt{2\pi T^2}} \int_{-\infty  }^{\infty} e^{- \frac{s'^2}{2T^2} }   ds'} =1 .
 \end{equation}
 For the ease of notation, we will skip $'$ in the following. 
 Using the binomial expansion 
 $(x + y)^\alpha = \sum_{\beta=0}^\alpha \binom{\alpha}{\beta} x^{\alpha - \beta}y^\beta$,
 Eqn.~\eqref{equ_ralpha} becomes
  \begin{equation}
  \R^m_\alpha =  \sum_{\beta = 0}^\alpha \binom{\alpha}{\beta} 
   \left(\frac{1}{\sqrt{2\pi T^2}} \int_{-\infty}^{\infty}e^{\frac{- s^2}{2T^2}} s^{\alpha - \beta} ds \right) (ic_m T^2)^{\beta} . 
 \end{equation}
 The statement follows by using expression~\eqref{equ_I} for the term in brackets.
 \end{proof}
 
 Based on these results, we can further expand \eqref{eqn_sum2} to get
 \begin{equation}\label{eqn_sum3}
  \begin{split}
    \sum_{k = 0}^p \M_{jk}  \pp{}{t}  \V_k(t)  =  
    \sum_{m} \sum_{k = 0}^p \sum_{l = 0}^p \Fu_{m,k,l}(t)  \R_{jkl}^m e^{ic_m t} e^{- \frac{c_m^2 T^2}{2} } 
     \quad \text{for} \quad  j = 0,1,2,
  \end{split}
 \end{equation}
 where 
 $\M_{jk} = \M_\alpha$ for $ \alpha = j+k$ and 
 $\R_{jkl}^m = \R^m_\alpha $ for $ \alpha = j+k+l$.
 Note that $\Fu_{m,k,l} =\Fu_{m,l,k} $ for all $m,k,l$. 
 
 Equation~\eqref{eqn_sum3} will serve us in Section~\ref{sec:swingspring} 
 to develop a higher order averaging method for an ODE describing a 
 swinging spring. 
 But before we discuss this example, let us first present an explicit 
 representation of \eqref{eqn_sum3} and how the equation 
 behaves in the zero and infinity limits of $T$.

 \subsubsection{Explicit representation for polynomial order $p=2$}
 For the sake of clarity, we present an explicit representation of Eqn.~\eqref{eqn_sum3}
 for $p = 2$, i.e. $j,k,l = 0,1,2$. Proposition~\ref{prop_I}
 leads to a mass matrix $\M$ with coefficients (in terms of $\M_\alpha$):
 \begin{equation}
 \M_0 = 1 , \quad \M_1 = 0, \quad \M_2 = T^2 , \quad \M_3 = 0 ,\quad  \M_4 = 3T^4 ; 
 \end{equation}
 and we can explicitly write the inverse of $\M$ as
 \begin{equation}
 \M^{-1} = 
 \begin{pmatrix}
 1 & 0 & T^2  \\
 0 & T^2   & 0 \\
 T^2   & 0 & 3T^4   \\
 \end{pmatrix}^{-1} = 
    \begin{pmatrix}
    \frac{3}{2} & 0 & \frac{-1}{2T^2}  \\
    0 & \frac{1}{T^2}   & 0 \\
    \frac{-1}{2T^2}  & 0 & \frac{1}{2T^4}   \\
    \end{pmatrix} .
    \end{equation}
 Moreover, Proposition~\ref{prop_R} leads to the coefficients (in terms of $\R^m_\alpha$):
 \begin{equation}
 \begin{split}
  & \R_0^m =  1 , \quad \R_1^m =  ic_m T^2, \quad \R_2^m = T^2 - c_m^2 T^4, \quad \R_3^m =  i 3 c_m T^4 - i c_m^3 T^6, \\
  & \R_4^m = 3T^4 - 6 c_m^2 T^6 +  c_m^4 T^8, \quad \R_5^m =  i 15  cm T^6 - i 10 c_m^3  T^8 +    i c_m^5 T^{10},\\
 & \R_6^m = 15T^6  - 45 c_m^2 T^8    + 15 c_m^4 T^{10}  -  c_m^6 T^{12}.
 \end{split}
 \end{equation}
 With these coefficients, the explicit form of Eqn.~\eqref{eqn_sum3} reads
 \begin{equation}\label{eqn_mass_1}
 \begin{split}
& \begin{pmatrix}
M_{00} & 0 & M_{02} \\
0 & M_{11} & 0 \\
M_{20} & 0 & M_{22} \\
\end{pmatrix}
\begin{pmatrix}
 \dot\V_0(t)\\
 \dot\V_1(t)\\
 \dot\V_2(t)\\
\end{pmatrix} =
 \sum_{m} \sum_{k = 0}^2 \sum_{l = 0}^2 e^{ic_m t} e^{- \frac{c_m^2 T^2}{2} }
 \Fu_{m,k,l}  
 \begin{pmatrix}
  \R_{0kl}^m \\
  \R_{1kl}^m \\
  \R_{2kl}^m \\
 \end{pmatrix}
   \\
  & =  \sum_{m} e^{ic_m t} e^{- \frac{c_m^2 T^2}{2} } \times 
  \Big[ \Fu_{m,0,0} 
\begin{pmatrix}
  \R_{000}^m \\
  \R_{100}^m \\
  \R_{200}^m \\
 \end{pmatrix}
 + \Fu_{m,0,1} 
 \begin{pmatrix}
  \R_{001}^m \\
  \R_{101}^m \\
  \R_{201}^m \\
 \end{pmatrix}
 + \Fu_{m,0,2} 
 \begin{pmatrix}
  \R_{002}^m \\
  \R_{102}^m \\
  \R_{202}^m \\
 \end{pmatrix}
  +\Fu_{m,1,0} 
\begin{pmatrix}
  \R_{010}^m \\
  \R_{110}^m \\
  \R_{210}^m \\
 \end{pmatrix}
 \\ &
 + \Fu_{m,1,1} 
 \begin{pmatrix}
  \R_{011}^m \\
  \R_{111}^m \\
  \R_{211}^m \\
 \end{pmatrix}
 + \Fu_{m,1,2} 
 \begin{pmatrix}
  \R_{012}^m \\
  \R_{112}^m \\
  \R_{212}^m \\
 \end{pmatrix}
 +\Fu_{m,2,0} 
\begin{pmatrix}
  \R_{020}^m \\
  \R_{120}^m \\
  \R_{220}^m \\
 \end{pmatrix}
 + \Fu_{m,2,1} 
 \begin{pmatrix}
  \R_{021}^m \\
  \R_{121}^m \\
  \R_{221}^m \\
 \end{pmatrix}
 + \Fu_{m,2,2} 
 \begin{pmatrix}
  \R_{022}^m \\
  \R_{122}^m \\
  \R_{222}^m \\
 \end{pmatrix} 
 \Big]   .
  \end{split}
 \end{equation}
 Bringing $\M$ to the LHS, we result in
 \begin{equation}\label{eqn_mass_2}
 \begin{split}  
 \begin{pmatrix}
 \dot\V_0(t)\\
 \dot\V_1(t)\\
 \dot\V_2(t)\\
\end{pmatrix} =
 \sum_{m} \sum_{k = 0}^2 \sum_{l = 0}^2   e^{ic_m t} e^{- \frac{c_m^2 T^2}{2} }  
 \Fu_{m,k,l} 
 \left(
 \begin{pmatrix}
\frac{3}{2}  \\
0   \\
\frac{-1}{2T^2}    \\
\end{pmatrix}
\R_{0kl}^m 
 +
\begin{pmatrix}
  0   \\
 \frac{1}{T^2}     \\
  0     \\
\end{pmatrix}
 \R_{1kl}^m 
+
\begin{pmatrix}
 \frac{-1}{2T^2}  \\
  0 \\
  \frac{1}{2T^4}   \\
\end{pmatrix}
\R_{2kl}^m \right) .
 \end{split}
 \end{equation}
 As it is useful for the asymptotic study further below, we present a
 fully explicit version of \eqref{eqn_mass_2} in terms of tendencies
 for $\V_0$, $\V_1$, and $\V_2$ in Appendix~\ref{append_explicit}.
 This explicit representation illustrates clearly the coupling between
 the principal $\V_0$ term and the higher order ones $\V_i, i =1,\dots
 p$ and allows us to visualize the expansion's asymptotic behavior
 with respect to the averaging window $T$.  Here we explicitly see the
 smoothing effect for large averaging window $T$, as all of the terms
 are exponentially small in $T$ except for resonant cases where $c_m$
 is small.

\subsubsection{Asymptotic behaviour with respect to $T$}
\label{sec:asymptotics}

We illustrate the asymptotic behavior of the expansion~\eqref{eqn_mass_2} on the 
example presented above for $p=2$, but the results hold for general $p$.
In particular, we investigate its behaviour for large and small $T$,
which can be inferred from an explicit representation of \eqref{eqn_mass_2}, as 
shown in Appendix~\ref{append_explicit}.
Here, we present a general discussion while showing a concrete example in Section~\ref{sec:swingspring}.

\paragraph{Limit $T\rightarrow 0$:} 
We study equation~\eqref{eqn_mass_2} in the limit $T \rightarrow 0$.
In this limit, many terms vanish (cf. Appendix~\ref{append_explicit}) and we arrive at
\begin{equation}\label{eqn_limit_T0}
 \begin{split}
  \begin{pmatrix}
 \dot\V_0(t)\\
 \dot\V_1(t)\\
 \dot\V_2(t)\\
\end{pmatrix}
= \sum_{m}  e^{ic_m t}   \times  
\begin{pmatrix}
\Fu_{m,0,0}    \\
i c_m \Fu_{m,0,0}  + 2 \Fu_{m,0,1}       \\
- \frac{1}{2}c_m^2\Fu_{m,0,0}  + \Fu_{m,1,1} + 2 i c_m \Fu_{m,0,1}  + 2 \Fu_{m,0,2}      \\
\end{pmatrix} .
 \end{split}
 \end{equation}
 In the limit $T \rightarrow 0$ the higher order coefficients thus
 decouple from the lowest order coefficient $\V_0$, i.e. $\V_0$ does
 not depend on the higher order terms. When considering only $\V_0$, which
 is all that we are interested in, we recover in fact the non-averaged
 equation~\eqref{equ_gen_V}.

 \begin{remark}
 Considering the full system \eqref{eqn_limit_T0}, higher order terms
 do not vanish and in the numerical algorithm below, these terms will
 be calculated too.  However, for $T$ very small, the condition number
 of matrix $\M$ is very large leading to increased numerical errors
 (cf. discussion to Figure~\ref{fig_2d_pT0}).  In practice this can be
 dealt with by choosing a better conditioned basis such as an
 orthogonal basis with respect to the $\rho$-weighted inner product
 (which would be Hermite polynomials in the examples we have computed
 here).  However, this only becomes an issue in the small $T$ limit,
 and we are interested in larger $T$ values to facilitate larger
 timesteps in numerical integrations.\end{remark}

\paragraph{Limit $T\rightarrow \infty$:} We study expansion~\eqref{eqn_mass_2}
for $T\rightarrow \infty$. To this end, we have to 
distinguish between the terms in the sum on the RHS of \eqref{eqn_mass_2} with 
 coefficients (i) $c_m \neq 0$ and (ii) $c_m =  0$ 
 (see e.g. the coefficients in Table~\ref{table_coeff}).
 
 Case 1: for $c_m \neq 0$ there follows $e^{ic_m t} e^{\frac{-c_m^2 T^2}{2}} =0 $ for $T\rightarrow \infty$
 such that the corresponding terms on the RHS of \eqref{eqn_mass_2} are zero (cf. Appendix~\ref{append_explicit}). 
 
 Case 2: for $c_m =  0$ there follows $e^{ic_m t} e^{\frac{-c_m^2 T^2}{2}} =1 $ for any $T$.
 Hence, not all terms on the RHS vanish (cf. Appendix~\ref{append_explicit}) and we arrive at the following equation:
\begin{equation}\label{eqn_limit_Tinfty}
 \begin{split}
  \begin{pmatrix}
 \dot\V_0(t)\\
 \dot\V_1(t)\\
 \dot\V_2(t)\\
\end{pmatrix}
= \sum_{m}    
\begin{pmatrix} 
\Fu_{m,0,0}   -  3 T^4\Fu_{m,2,2}\\
2 \Fu_{m,0,1} + 6 T^2\Fu_{m,1,2} \\
 \Fu_{m,1,1} + 2 \Fu_{m,0,2} + 6 T^2\Fu_{m,2,2} 
\end{pmatrix}  .
 \end{split}
 \end{equation}
Similar to the zero limit \eqref{eqn_limit_T0}, there is a decoupling of the higher order modes from the principle one:
the $\Fu_{m,0,0}$ term vanishes in the higher order modes. Therefore, as all $\V_i,i=1,..p$ are 
initialized with zero, they remain zero including the $\Fu_{m,2,2}$ that would contribute 
otherwise to $\V_0$. 
Therefore, we recover indeed in this limit an asymptotic approximation of \eqref{equ_gen_V}
in form of \citet{majda1998averaging}.

 \section{Example ODE: swinging spring}
 \label{sec:swingspring}
 
 We illustrate the properties of the higher order averaging method introduced above 
 on an ODE that describes the dynamics of a swinging spring, 
 a model due to Peter Lynch \citep{holm2002stepwise}. 
 Although idealized, this model shows an interesting analogy to geophysical flows as it 
 exhibits a high sensitivity of small scale oscillation on the large scale dynamics.

 \subsection{The swinging spring}
 \label{sec_ode}
 
 For the swinging spring (expanding pendulum) model of Holm and Lynch~\citep{holm2002stepwise}, we consider 
 the following equations of motion
 \begin{equation}\label{ode}
 \begin{split}
  \ddot x(t) + \omega_R^2 x(t) = \lambda x(t) z(t) ,\\
  \ddot y(t) + \omega_R^2 y(t) = \lambda y(t) z(t) ,\\
  \ddot z(t) + \omega_Z^2 z(t) = 0.5 \lambda (x(t)^2 + y(t)^2) ,
 \end{split}
 \end{equation}
 where $x(t),y(t),z(t)$ are (time dependent) Cartesian coordinates centered at the point of equilibrium, 
 $\omega_R = \sqrt{g/l}$ is the frequency of linear pendular motion, and 
 $\omega_Z = \sqrt{k/m_u}$ is the frequency of elastic oscillation 
 with spring constant $k$, unit mass $m_u$ and gravitational constant $g$.
 $l_0$ denotes the unstretched length while $l$ is the length at equilibrium 
 such that $\lambda = l_0 \omega_Z^2 / l^2$. The dot denotes the time derivative
 $\frac{d}{dt}$.

 To bring Eqn.~\eqref{ode} into the abstract form~\eqref{equ_gen_U}, 
 we first reformulate the former into the first order system 
 \begin{equation}\label{ode1st}
 \begin{split}
 & \dot x(t) - \omega_R p_x(t) = 0, \quad \dot p_x(t) + \omega_R x(t) = \frac{\lambda}{\omega_R} x(t) z(t) ,\\
 & \dot y(t) - \omega_R p_y(t) = 0, \quad \dot p_y(t) + \omega_R y(t) = \frac{\lambda}{\omega_R} y(t) z(t) ,\\
 & \dot z(t) - \omega_Z p_z(t) = 0, \quad \dot p_z(t) + \omega_Z z(t) = \frac{\lambda}{\omega_Z} (x(t)^2 + y(t)^2),
 \end{split}
 \end{equation}
 with $p_x := \pp{p}{x}$ and so forth. Then, by introducing the complex numbers 
 $ X(t) = x(t) + i p_x(t) , \ Y(t) = y(t) + i p_y(t) , \  Z(t) = z(t) + ip_z(t) , $
 we can rewrite Eqn.~\eqref{ode1st} as
 \begin{align}
  \dot X(t) + i\omega_R X(t) = i \frac{\lambda}{\omega_R}  Re(X(t)) Re(Z(t)) ,\\
  \dot Y(t) + i\omega_R Y(t) = i \frac{\lambda}{\omega_R}  Re(Y(t)) Re(Z(t)) ,\\
  \dot Z(t) + i\omega_Z Z(t) = i \frac{\lambda}{2\omega_Z} (Re(X(t))^2 + Re(Y(t))^2) .
 \end{align}
 This form of the equations can be expressed in the abstract 
 form~\eqref{equ_gen_U} by defining the velocity as 
 \begin{equation}
 \U(t) = \left(\begin{array}{c}
            X(t) \\ Y(t) \\ Z(t)
           \end{array}
    \right) = 
\left(\begin{array}{c}
 x(t) + ip_x(t)\\ y(t) + ip_y(t)\\ z(t) + ip_z(t) 
 \end{array}\right) ,
\end{equation}
 the linear operator as
 \begin{equation}\label{lin_op}
 L = \left(\begin{array}{c}
            - i \omega_R \\ 
             -i \omega_R \\ 
             -i\omega_R \rho
           \end{array}
    \right)  ,
 \end{equation}
 and the nonlinear (quadratic) operator as
 \begin{equation}\label{nonlin_op}
 N(\U(t),\U(t)) = i \frac{ \lambda}{ \omega_R} 
 \left(
 \begin{array}{c}
    Re(X(t))Re(Z(t)) \\ 
    Re(Y(t))Re(Z(t))  \\  
    \frac{\rho}{2}  (Re(X(t))^2 + Re(Y(t))^2)  
 \end{array}
 \right) .
 \end{equation} 
 
 Using these definitions, Eqn.~\eqref{ode} can also be written in the 
 modulated form~\eqref{equ_gen_V}. 
 That is, the time evolution of the velocity $\V(t)$ with transformation
 \begin{equation}
 \V(t) = e^{-tL} \U(t) , \quad
 \V(t) = \left(\begin{array}{c}
            \hat X(t) \\ \hat Y(t) \\ \hat Z(t)
           \end{array}
    \right)  ,
\end{equation}
under the linear operator $L$ of \eqref{lin_op} is given by Eqn.~\eqref{equ_gen_V} 
with nonlinear operator $N$ of \eqref{nonlin_op}. 
Explicitly, we arrive at the component equations:
 \begin{align}
\dot{\hat X} & = c_{xy}  \hat X \hat Z   e^{- i \omega_R \rho t}
             + c_{xy}  \hat X \overline{\hat Z} e^{i \omega_R \rho t} 
             + c_{xy}  \hat Z \overline{\hat X} e^{- i \omega_R \rho t + 2 i \omega_R t} 
             + c_{xy}   \overline{\hat X} \overline{\hat Z} e^{i \omega_R \rho t + 2 i \omega_R t} \label{eqn_1}\\
 \dot{\hat Y} & = c_{xy}  \hat Y \hat Z   e^{- i \omega_R \rho t}
             + c_{xy}  \hat Y \overline{\hat Z} e^{i \omega_R \rho t} 
             + c_{xy}  \hat Z \overline{\hat Y} e^{- i \omega_R \rho t + 2 i \omega_R t} 
             + c_{xy}   \overline{\hat Y} \overline{\hat Z} e^{i \omega_R \rho t + 2 i \omega_R t} \label{eqn_2}\\  
 \dot{\hat Z} & = c_z (\hat X^{2} + \hat Y^{2})  e^{i \omega_R \rho t - 2 i \omega_R t}
                + 2c_z (\hat X  \overline{\hat X} + \hat Y \overline{\hat Y} ) e^{i \omega_R \rho t}
                + c_z  (\overline{\hat X}^{2} +  \overline{\hat Y}^{2} )e^{i \omega_R \rho t + 2 i \omega_R t} \label{eqn_3}                  
\end{align}
with coefficients $c_{xy} = 3  i  \omega_R\rho^2/16 $ and $c_z    = 3  i  \omega_R\rho/32$. 
Overline notation denotes the conjugate. 
For ease of notation, we omitted here to include the time dependency of the transformed components
$\hat X,\hat Y, \hat Z$. 
After solving for $\V$, we have to transform back 
to $\U$ with $\U(t) = e^{tL} \V(t)$ to find solutions of \eqref{equ_gen_U}.

The explicit formulation in \eqref{eqn_1}--\eqref{eqn_3} 
allows us to identify the problem (or ODE) dependent coefficient $c_m$ and 
$\Fu_{m,k,l}= (F^x_m,F^y_m,F^z_m) \ \forall k,l$ of the expansion~\eqref{eqn_sum3}. 
Note that the latter coefficients change with $m$ while they are the same 
for all $k,l$. The following table summarizes these coefficients:
 \begin{table}[h]\centering
\begin{tabular}{ l| l | l | l}
   $c_m$ & $F^x_m$ &  $F^y_m$ & $F^z_m$ \\\hline
  $c_1:= (\rho - 2)\omega_R$ & 0 & 0 & $F^z_1 = c_z  (\hat X \hat X + \hat Y\hat Y) $ \\
  $c_2:= \rho \omega_R$ & $F^x_2 = c_{xy}  \hat X \overline{\hat Z} $ & $F^y_2 = c_{xy}  \hat Y \overline{\hat Z} $ & $F^z_2 = 2c_z (\hat X \overline{\hat X} + \hat Y \overline{\hat Y})$ \\
  $c_3: =  (\rho + 2)\omega_R$ & $F^x_3 = c_{xy}  \overline{\hat X}  \overline{\hat Z} $ & $F^y_3 = c_{xy}  \overline{\hat Y}  \overline{\hat Z} $ &
      $F^z_3 = c_z   (\overline{\hat X} \overline{\hat X} + \overline{\hat Y}  \overline{\hat Y})$ \\
  $c_4 := - \rho \omega_R$ &$F^x_4 = c_{xy}  \hat X  \hat Z $ & $F^y_4 = c_{xy}  \hat Y  \hat Z $ & 0 \\
  $c_5 := -(\rho + 2)\omega_R$ & $F^x_5 = c_{xy}  \overline{\hat X} \hat Z $  & $F^y_5 = c_{xy}  \overline{\hat Y} \hat Z $ & 0 \\
\end{tabular}
\caption{Coefficients of expansion~\eqref{eqn_sum3} for the swinging spring model.}
\label{table_coeff}
\end{table}
 
 The coefficients from Table~\ref{table_coeff} are used in expansion~\eqref{eqn_sum3} to obtain 
 the higher order phase averaged system.

 \paragraph{Limits for $T$:} Similarly to Section~\ref{sec:expint}
 but here for concrete model equations, i.e. 
 for the ODE \eqref{ode} of the swinging spring, 
 we present the equations describing the time evolution of $\V_0$ 
 resulting from taking the limits of \eqref{eqn_sum3} in $T$,
 i.e. for (i) $T \rightarrow 0$ and (ii) $T \rightarrow \infty$.

 Case (i): as discussed in \eqref{eqn_limit_T0}, in the $T$ to zero limit
 the higher order terms decouple from the principle mode $\V_0$ such that 
 for the time evolution of the latter only the $\Fu_{m,0,0}$ terms contribute. 
 In this limit and using the coefficients of Table~\ref{table_coeff}, 
 the resulting set of equations for $\V_0$
 coincides with equations \eqref{eqn_1}--\eqref{eqn_3} of the exact model.

 Case (ii): as shown in \eqref{eqn_limit_Tinfty}, 
 all terms on the RHS of the expansion~\eqref{eqn_sum3} vanish in the $T \rightarrow \infty$ 
 limit except of those with $c_m = 0$ (i.e. here for $\rho =2$). Considering further that 
 higher order terms in $\V$ (i.e. for $k>0$ and/or $l>0$) are initialized with zero, 
 these terms, and in particular the $\Fu_{m,2,2}$ terms in the first line in \eqref{eqn_limit_Tinfty}, 
 remain zero. 
 Then, for any approximation order $p$, the resulting set of equations 
 for $\V_0$ are given by 
    \begin{equation}\label{equ_avg_Tinfty}
    \begin{split}
    \dot{\hat X} = c_{xy}  \hat Z \overline{\hat X} , 
    \qquad 
    \dot{\hat Y}  =   c_{xy}  \hat Z \overline{\hat Y} , 
    \qquad
    \dot{\hat Z} = c_z (\hat X^{2} + \hat Y^{2}) .         
    \end{split}
    \end{equation}
    This is exactly the phase-averaged model that was derived by
    \citet{holm2002stepwise} using Witham averaging.

 \section{Numerical results} 
 \label{sec:num}
 
 We investigate accuracy and stability of the higher order phase averaging method on numerical simulations
 of the swinging spring model as introduced in the previous section. 
 We compare solutions obtained from the approximated models with those 
 from the non-averaged (exact) equations~\eqref{equ_gen_V}. In particular,
 the convergence behavior of solutions obtained from averaging models with increasing approximation 
 order to exact solutions will be investigated for finite averaging windows. 
 Besides showing that higher order phase averaging indeed leads to a better approximation
 of the fast oscillations around the slow manifold than standard (first order) averaging, 
 we illustrate that this also reduces the drift from averaged solutions away from the 
 exact solutions, which we experienced in our simulations.

 Analogously to \citet{holm2002stepwise}, we choose for the numerical simulations the 
 following parameters for the ODE in equation~\eqref{ode}:
 $m_u = 1\,$kg, $l=1\,$m, $g= \pi^2\,$ m s$^{-2}$, $k = 4 \pi^2$ kg s$^{-2}$ 
 and hence $\omega_R = \pi$ and $\omega_Z = 2\pi$ with resonance condition $\omega_Z = 2\omega_R$.
 The unstretched length $l_0 = 0.75\,$m follows from the balance $k\, (l - l_0) = g\, m_u$ 
 allowing us to determine $\lambda$ in \eqref{ode}. Finally, the equations are 
 initialized with 
 \begin{equation}\notag
  (X_0(0),Y_0(0),Z_0(0)) = (0.006,\, 0,\, 0.012)\,\text{m} , \quad  (\dot X_0(0),\dot Y_0(0),\dot Z_0(0)) = (0,\, 0.00489,\, 0)\,\text{m s}^{-1},
 \end{equation}
 for the principle mode $\V_0$ while values for $\V_{\neq 0}$ and $\dot \V_{\neq 0}$ are all initialized with zero 
 and reset back to zero after an integration window $\Delta T$. 
 This resetting avoids potential instabilities in the higher order modes in $\V_i,i=1,...p$ 
 (cf. discussion to Figure~\ref{fig_Vxyz_1000s}). If not stated otherwise, we 
 use a resetting window of $\Delta T = 100\,$s while other choices are also possible 
 because this choice does not significantly impact on the accuracy of the solutions,
 as shown below in Figure~\ref{fig_Vx1}. 
 
 Note that with this parameter choice, 
 the fast and slow motions occurring in the simulations 
 have a scale separation at the order of about $\epsilon \approx 0.01$ (in good agreement with 
 values for geophysical flows).

 We integrate the ODEs in time using the adaptive RK4/5 explicit
 integrator implemented through the Python Scipy package. We integrate
 up to 1000 seconds (corresponding to 1000 vertical oscillations)
 using adaptive relative and absolute tolerance of $1.49012e^{-8}$ unless
 otherwise stated.
 
 To ease notation, we consistently denote also 
 solutions of the exact (equations~\eqref{eqn_1}--\eqref{eqn_3}) and the averaged ($p=0$) models 
 with $\U_0 = (X_0, Y_0, Z_0)$ and $\V_0 = (\hat X_0, \hat Y_0, \hat Z_0)$ even 
 though there are no higher order terms present in these cases.

 \begin{figure}[t]\centering
 \begin{tabular}{cc}
  \includegraphics[scale=.5]{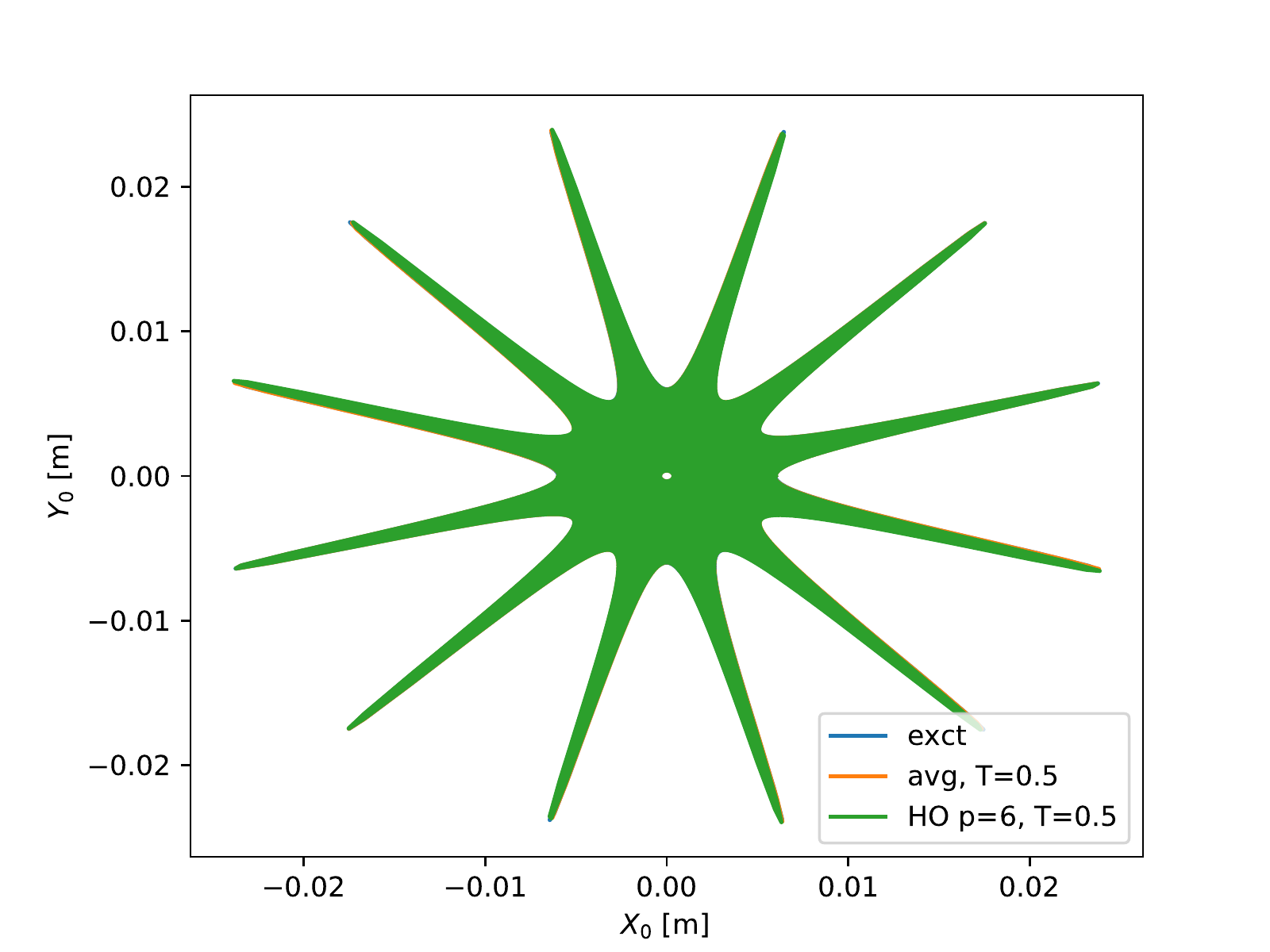} &
  \includegraphics[scale=.5]{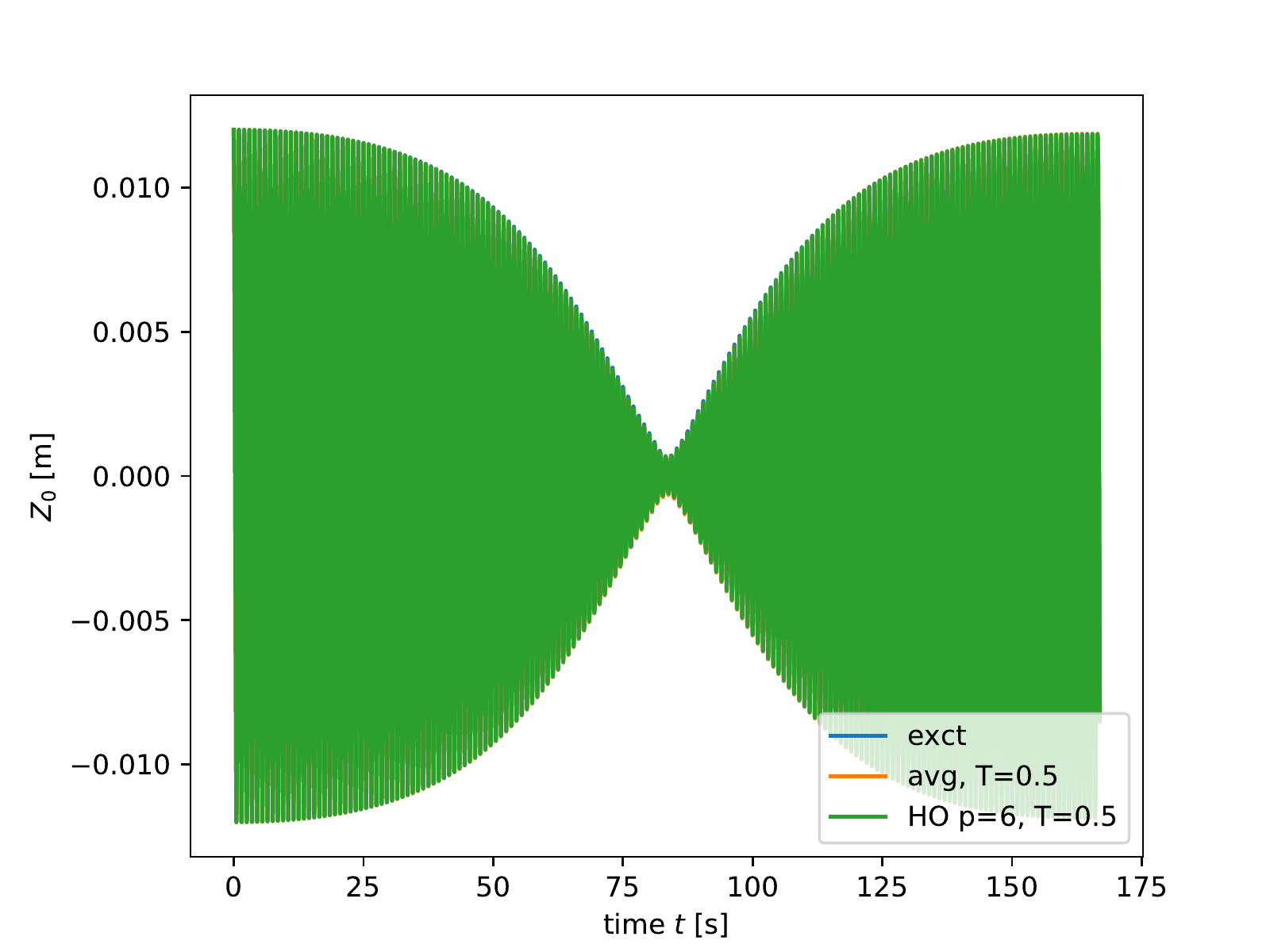} 
  \end{tabular}
  \caption{Left: horizontal projection of $X_0,Y_0$ components of pendulum movement 
  for $1000\,$s for the exact (exct), averaged (avg) and higher order (HO) averaged models.
  Right: Vertical projection of $Z_0$ component of pendulum movement over time $t$ 
  for $167\,$s (corresponding approximately to the first of six modulation cycles). 
  }
  \label{fig_xy_1000s}
  \end{figure}

 \subsection{Large scale dynamics, smoothing behavior and stability}

 First, we verify that the solutions of our phase averaged models 
 stay close to the corresponding exact solutions and capture the same characteristics typical to the swinging spring model (given, for instance, by the stepwise precision of the pendulum in the $(x,y)$-plane). 
 To this end, we compare for a long term simulation of $1000\,$s the solutions in $\U_0$ of the 
 exact model, the averaged one with $p=0$ that corresponds to \citet{peddle2019parareal} and a higher order (HO) version with, e.g., $p=6$. In order to guarantee stable simulations, we reset the 
 higher order values in $\U_i$ (or $\V_i$), $i=1,...p$, to zero, here after a resetting time of $\Delta T = 100\,$s. 
 The size of this resetting window $\Delta T$ does only marginally impact the accuracy (cf. 
 Figure~\ref{fig_Vx1}) but resetting prevents the averaged solutions to become unstable, 
 in particular for large $p$ and long simulations (cf. Figure~\ref{fig_Vxyz_1000s}).

 \paragraph{Large scale dynamics.}
 In Figure~\ref{fig_xy_1000s} (left) we see the projection of the pendulum movement onto the horizontal $(x,y)$-plane. 
 Both, the solutions for the averaged and HO averaged models are very 
 close to the exact solutions, almost indistinguishable in the given figures that show 
 the large scale dynamics. Here, we used an averaging window of $T = 0.5\,$s while for smaller
 $T$ the accuracy does increase (cf. e.g. Figure~\ref{fig_2d_pTmiddle}).
 In particular, the solutions of the averaged models capture very well the stepwise precision 
 of the swinging spring model.  
 In the right panel, the movement of the pendulum in the $Z$ axis versus time is shown. 
 Also here, both averaged and HO averaged solutions are very close to the exact one. 
 The figure illustrates nicely fast oscillations with a slowly varying amplitude envelope.
 This typical behavior of the swinging spring model makes it thus a perfect test problem for 
 our averaging method developed above.

 \begin{figure}[t]\centering
 \begin{tabular}{cc}
   \includegraphics[scale=.8]{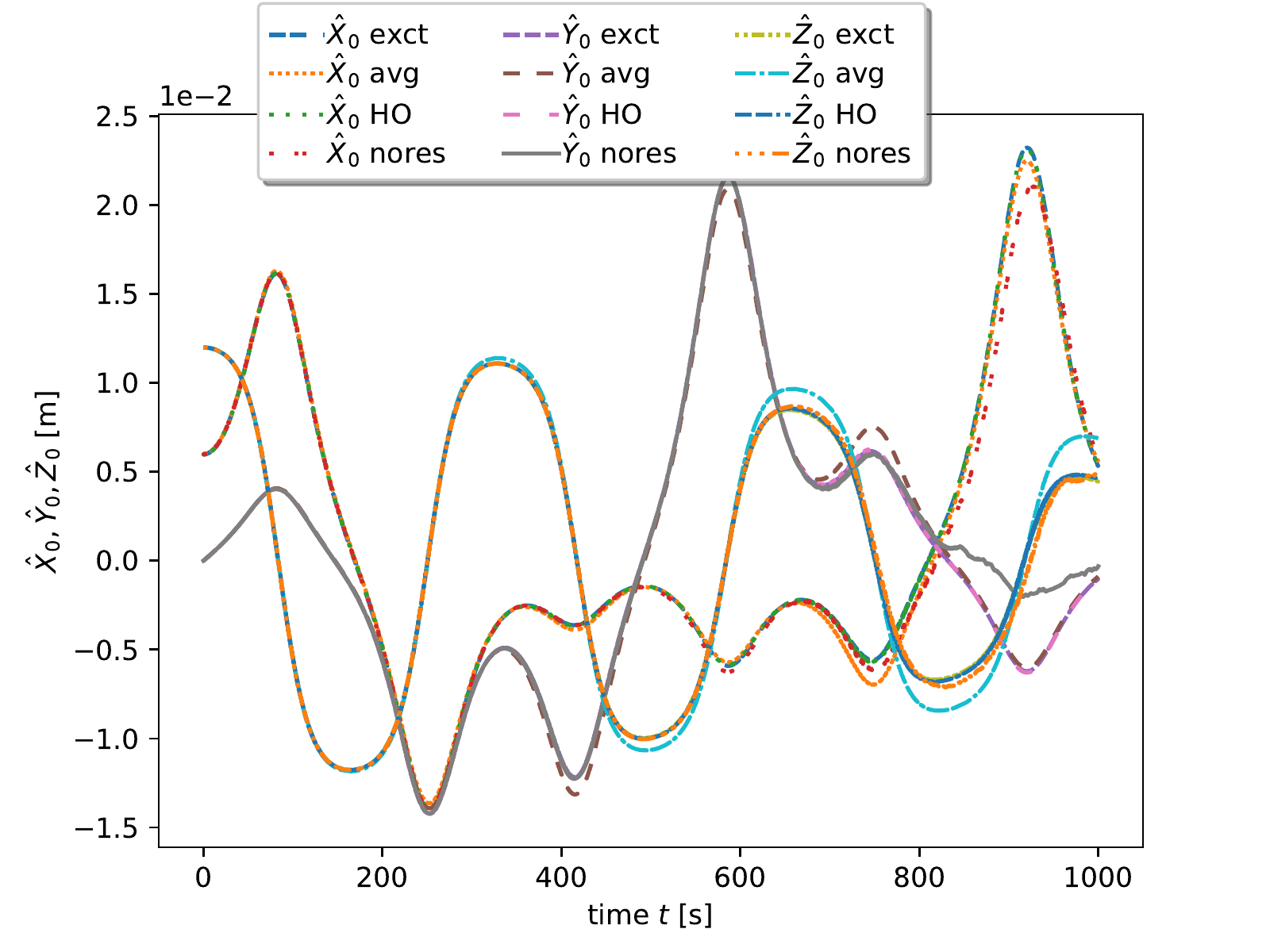} 
  \end{tabular}
  \caption{
  $\hat X_0,\hat Y_0, \hat Z_0$ solutions 
  of the exact (exct), averaged (avg with $p=0$) and 
  higher order (HO) models (with $p=8$) and $T=0.2\,$s; reset 
  every $\Delta T = 100\,$s or with no resetting (nores). 
  Here, no resetting leads to instabilities for high $p \geq 8$
  after about $800\,$s.
  }
  \label{fig_Vxyz_1000s}
  \end{figure}

 \paragraph{Smoothing properties.}

 Next, we illustrate how the proposed averaging algorithm filters 
 fast motions from the dynamics of the swinging spring model 
 while still accounting for their effect on the slow dynamics. 
 Given the factorization assumption~\eqref{factoring} that underlies
 equations~\eqref{equ_gen_V}, 
 it is useful to study in more detail solutions in terms of the 
 velocity $\V$ because the latter provide solutions of the slow dynamics 
 without most of the fast oscillations intrinsic to $\U$.
 As such, $\V$ (in particular $\V_0$) 
 allows us to clearly illustrate how the averaging method filters 
 fast oscillations and how the higher order methods better approximate 
 small scale oscillations around the slow dynamics.  
 
 Figure~\ref{fig_Vxyz_1000s} shows the solutions for 
 $\hat X_0,\hat Y_0, \hat Z_0$
 of the exact, the averaged, and the HO averaged models without resetting and with 
 resetting at $\Delta T = 100\,$s. 
 We use an averaging window of $T=0.2\,$s and $p=8$ for the HO method as this provides 
 sufficiently accurate solutions (cf. Figure~\ref{fig_2d_pTmiddle}). 
 The figure illustrates also a solutions of the HO model for $p=8$ 
 ($\hat X_0$ nores, $\hat Y_0$ nores, $\hat Z_0$ nores) that becomes slightly unstable 
 after about $800\,$s in case no resetting is used. This instability, however, 
 is easily controlled by simple resetting the higher order velocity component 
 at each $\Delta T$ to zero, resulting in the stable solutions 
 $\hat X_0$ HO, $\hat Y_0$ HO,  $\hat Z_0$ HO.

 \begin{figure}[t]\centering
 \begin{tabular}{cc}
    \includegraphics[scale=.4]{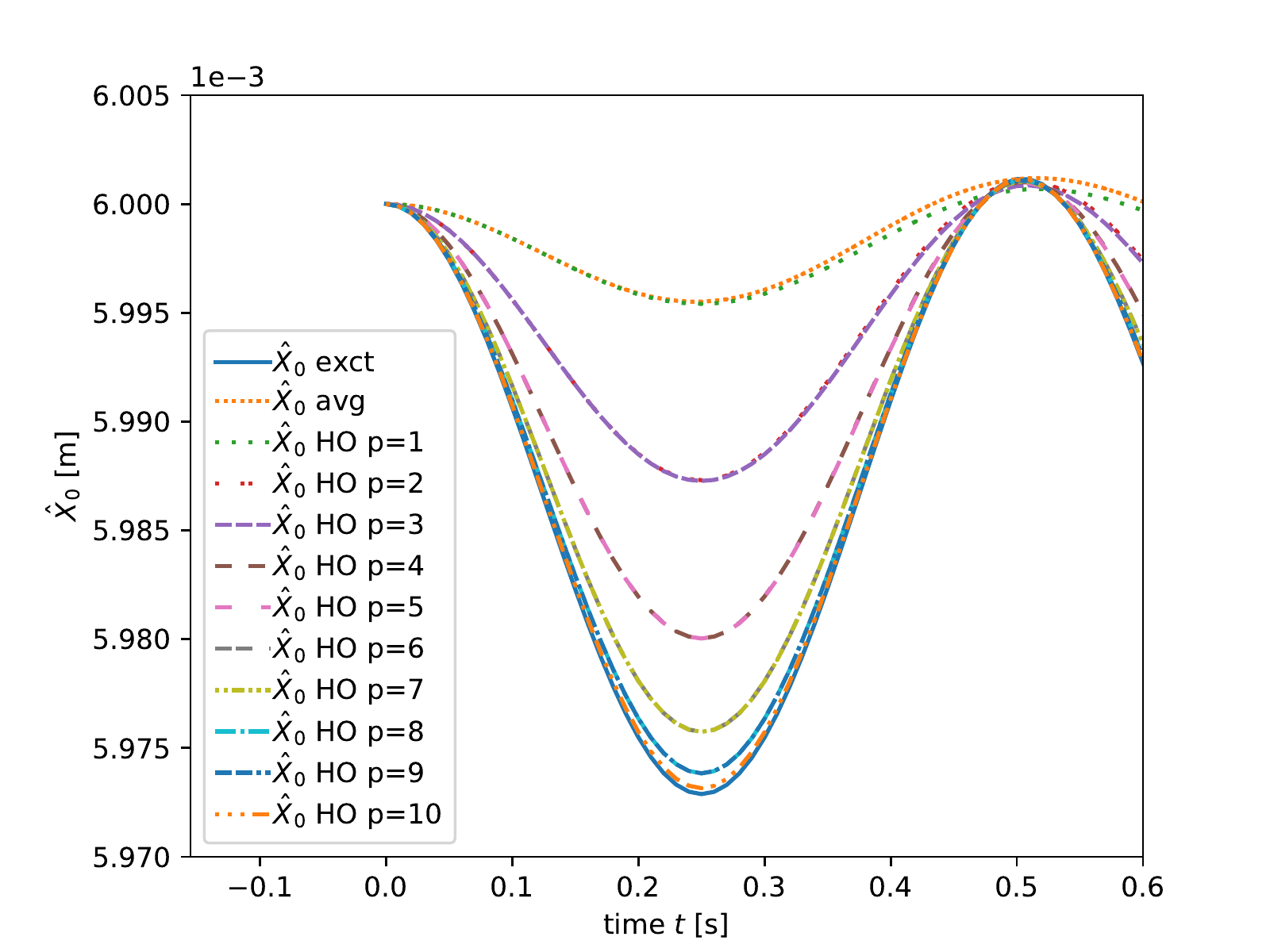} &
    \includegraphics[scale=.4]{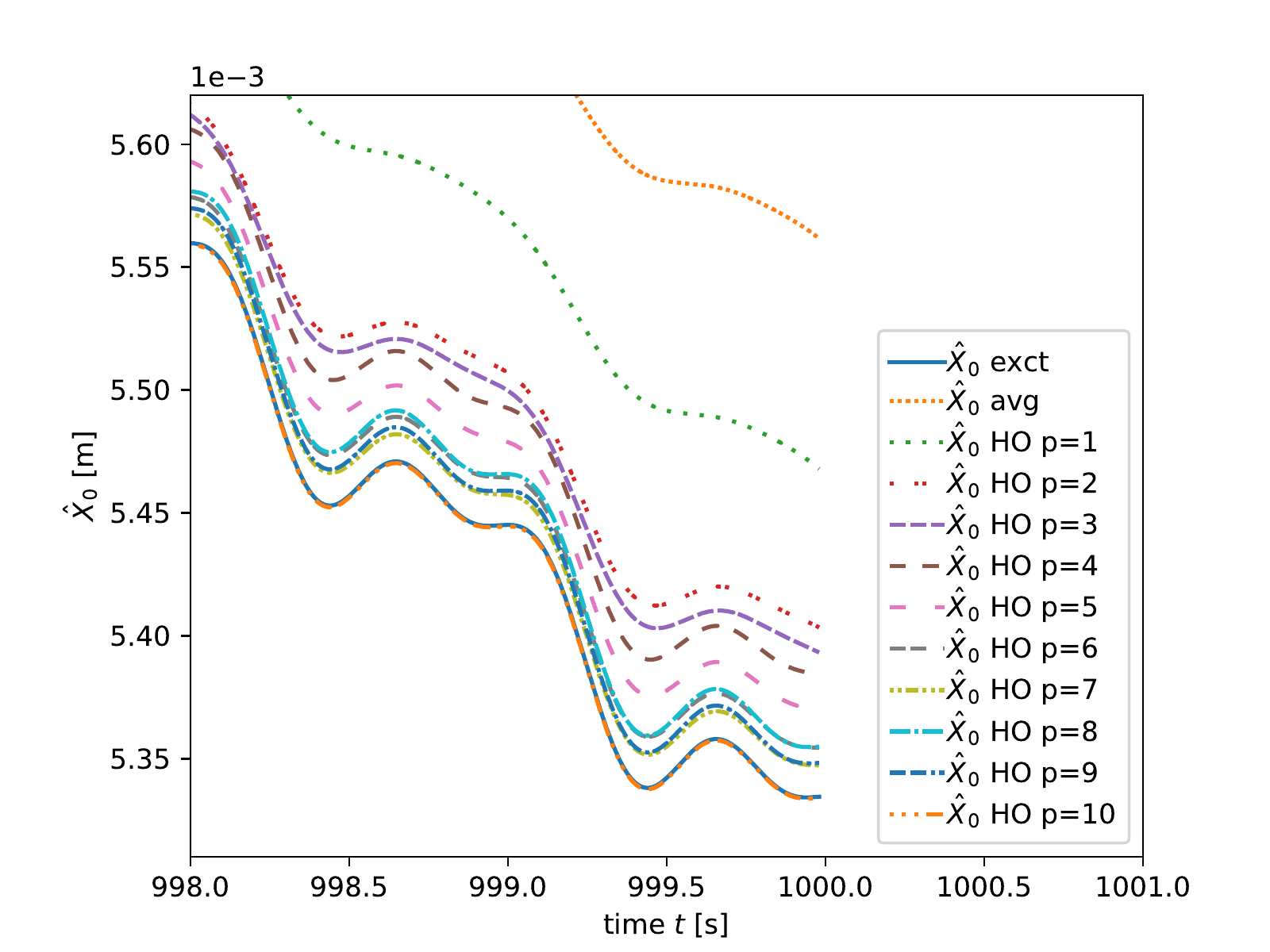} \\
    \includegraphics[scale=.4]{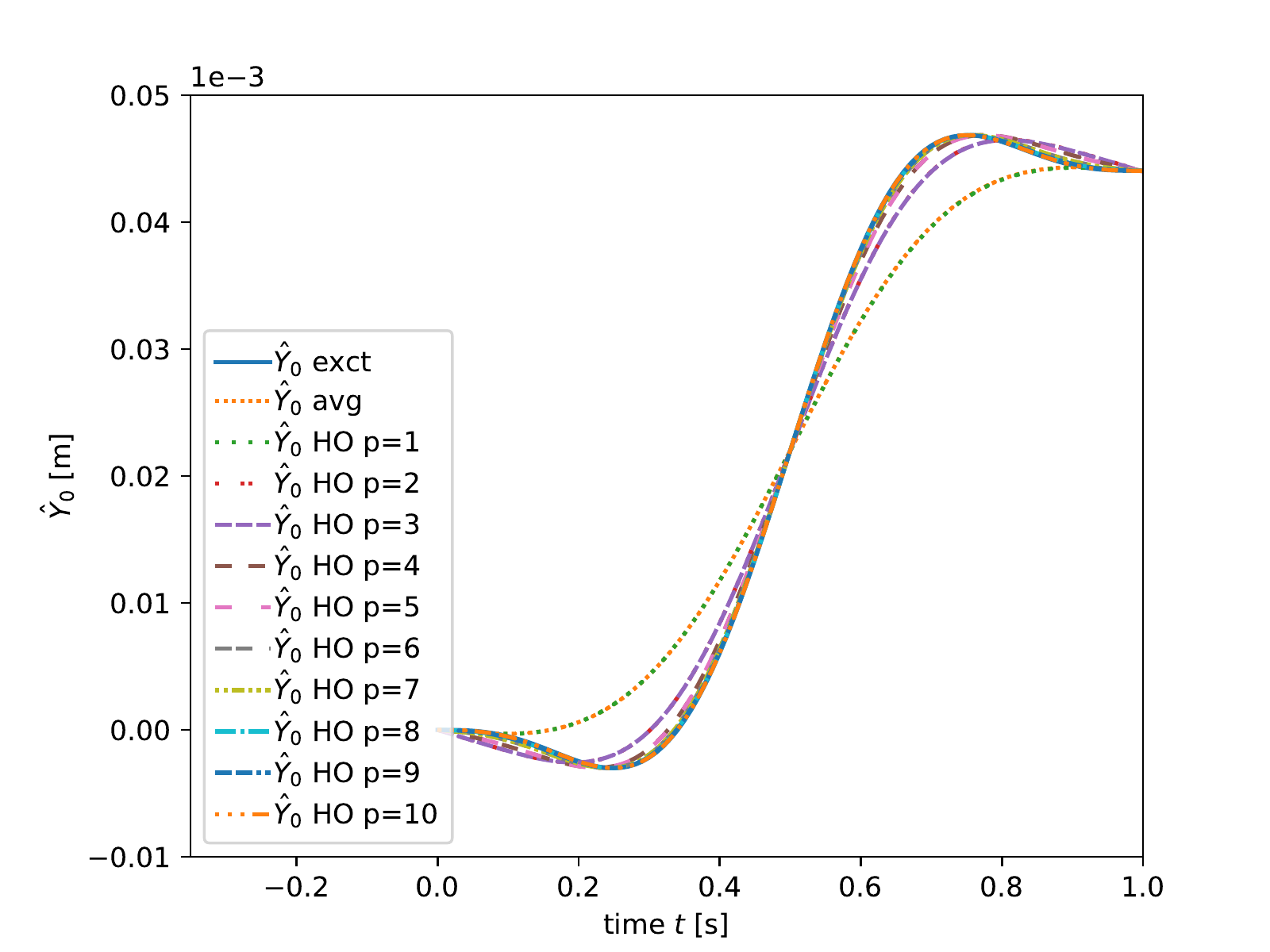} &
    \includegraphics[scale=.4]{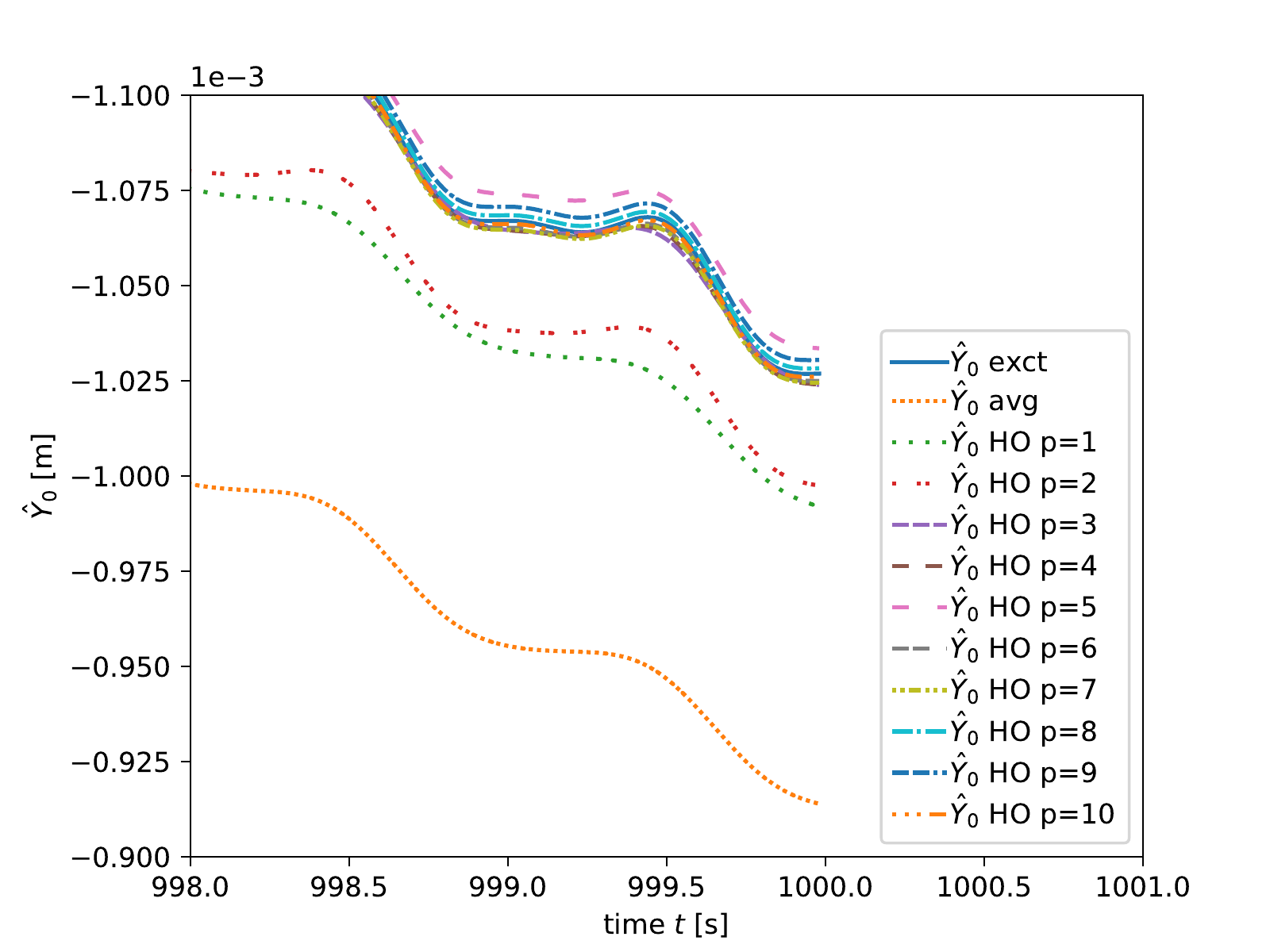} \\
    \includegraphics[scale=.4]{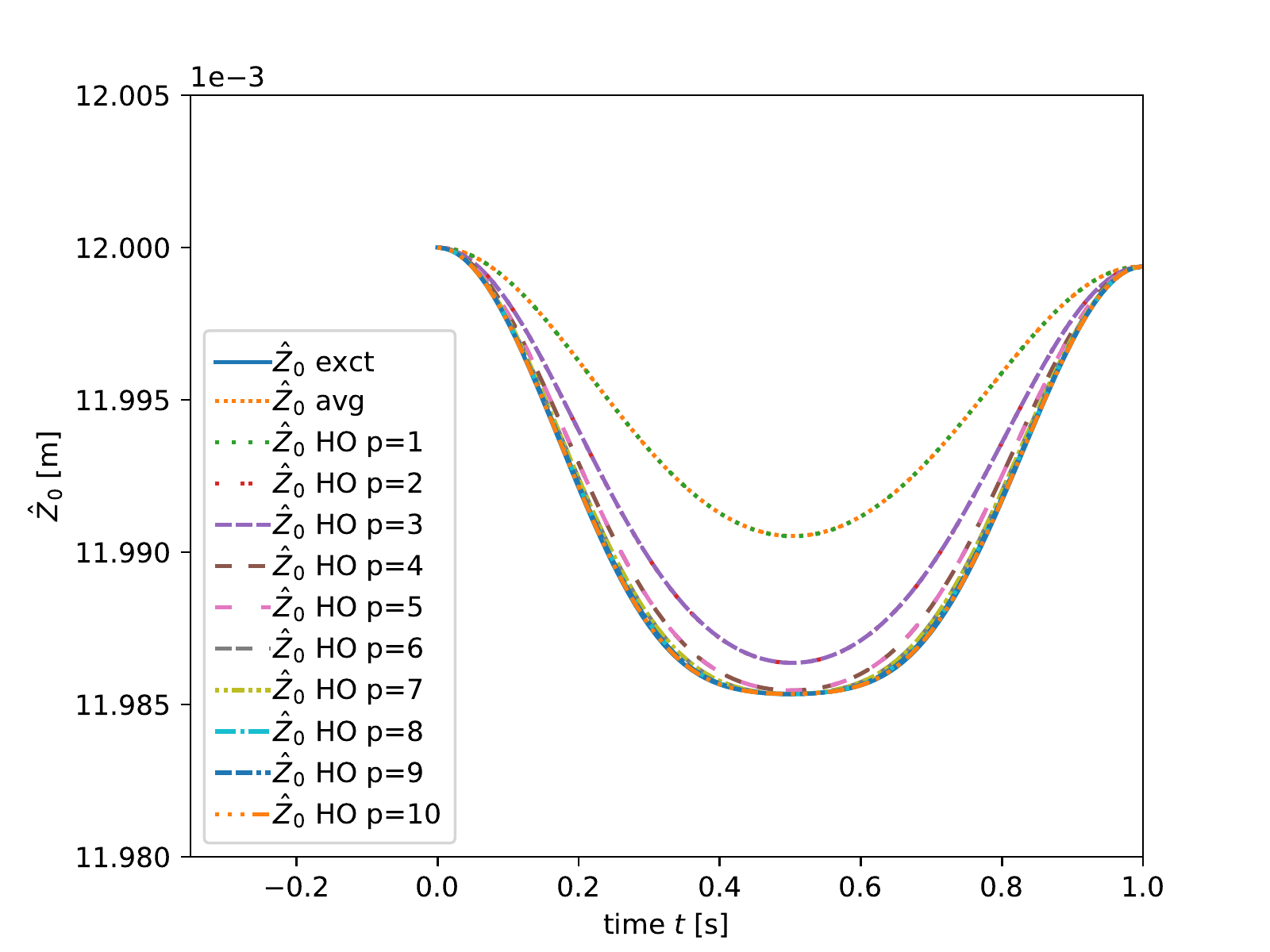} &
    \includegraphics[scale=.4]{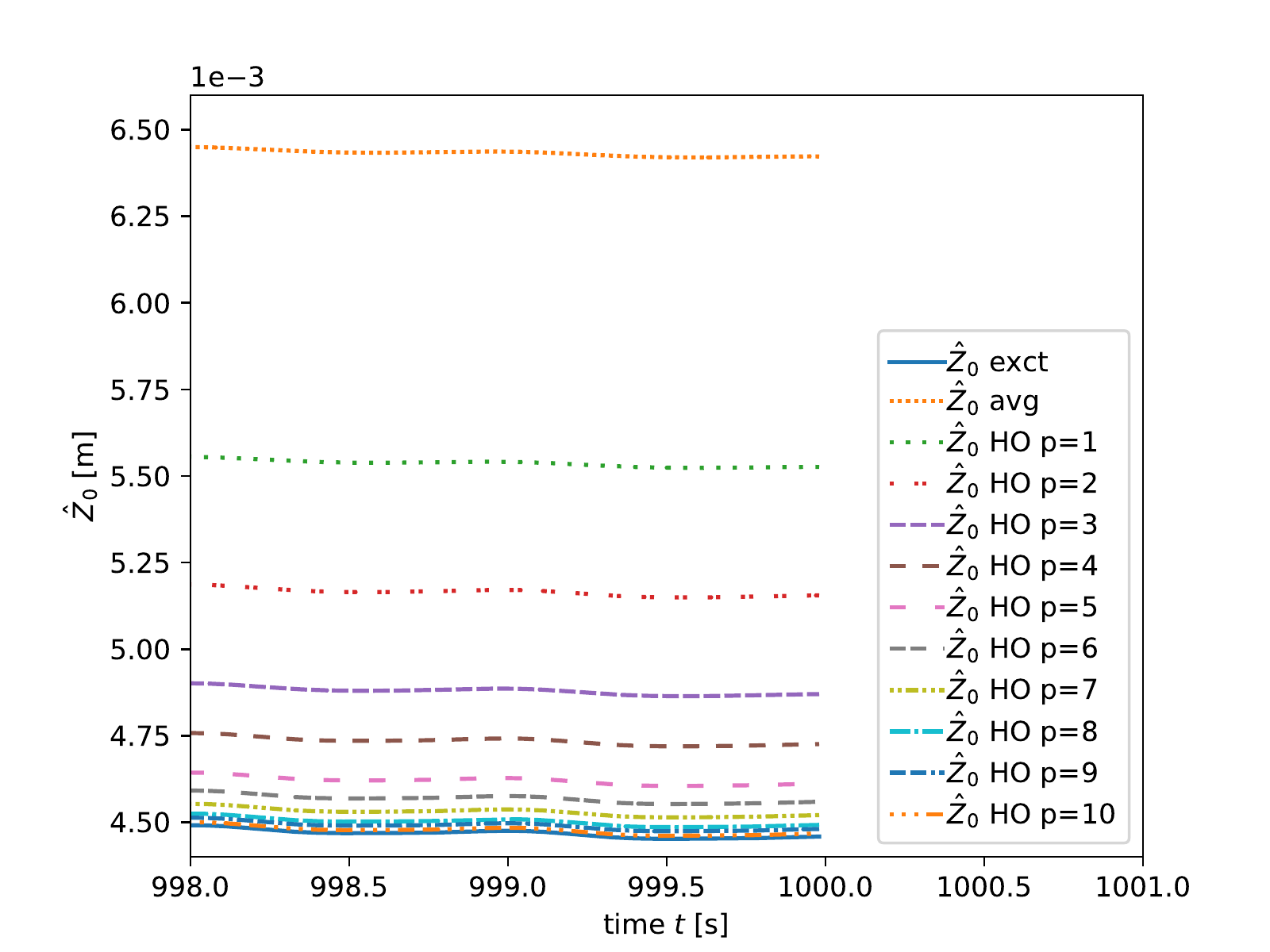}     
    \end{tabular}
  \caption{
  $\hat X_0,\hat Y_0, \hat Z_0$ solutions of the exact (exct), 
  averaged (avg with $p=0$) and higher order (HO) models
  with resetting ($\Delta T = 100\,$s)
  and averaging window $T = 0.15\,$s. 
  }
  \label{fig_Vxyz_1000s_XYZzoom}
  \end{figure}

 Next, we compare the exact solutions in all three components with those of the averaged models.
 We realize that for longer simulation times, the averaged solutions slightly drift away
 from the exact ones, as shown on the right panels in Figure~\ref{fig_Vxyz_1000s_XYZzoom}.
 This is particularly obvious in the $\hat Z_0$ component. Further in this figure, we see clearly that 
 with increasing approximation order $p$ the solutions' drift reduces significantly 
 whilst the small scale oscillations are better approximated. 
 Note that the standard averaging method with $p =0$ corresponds to the original version as introduced in 
 \citet{haut2014asymptotic}. As discussed therein, the averaging leads to some drift away 
 from the exact solutions. As the figure illustrates clearly, including the higher averaging
 terms significantly reduces this drift while providing solutions much closer to the 
 exact ones.

 When considering short integration times, as shown on 
 the left panels in Figure~\ref{fig_Vxyz_1000s_XYZzoom},
 it is particularly obvious that the higher order method 
 more accurately approximates small scale oscillations
 than the standard method. While both approximate well 
 the slower dynamics, the standard averaging method
 almost entirely smooths out the faster oscillations
 while with increasing polynomial approximation order, 
 these fast oscillations 
 are better captured by the solutions of the HO averaging models.

 \subsection{Convergence and accuracy of the higher order models}
 
 Here, we study the accuracy of solutions obtained from our higher order averaging method 
 compared to exact solutions, obtained from numerical simulations of 
 \eqref{equ_gen_V}, by performing a convergence analysis of the solutions and their 
 dependencies on the averaging window $T$ and the polynomial degree $p$. 
 For the chosen $T$ from above, we already verified that with increasing 
 $p$ the error in the averaged solutions decreases significantly. However,
 these error depend crucially on the choice of averaging window $T$. 
 For instance, in case $T$ is too large, even high polynomial approximations 
 will have an significant error.

 In the following, we will therefore study solution errors with respect to 
 both $p$ and $T$ and provide a 2 dimensional (2D) error map. 
 To determine these 2D error maps, we define the following $L_2$ error 
 measure for solution $\hat X(t)$ (and similarly for $\hat Y(t),\hat Z(t)$)
 \begin{equation}
 {L_2} \hat X(t) = \frac{\sqrt{\int_0^{tf} (\hat X(t) - \hat X_{e}(t))^2 dt}  }{\sqrt{\int_0^{tf} \hat X_{e}(t)^2 dt}}
 \end{equation}
 where $\hat X(t)$ and $\hat X_{e}(t)$ is the x-component of $\V$ of the averaged and exact solutions, 
 respectively (and similarly for $\hat Y, \hat Z$), and where $tf$ is the final time.

 \begin{figure}[h]\centering
 \begin{tabular}{ccc}
    \hspace{-1cm} \includegraphics[scale=.42]{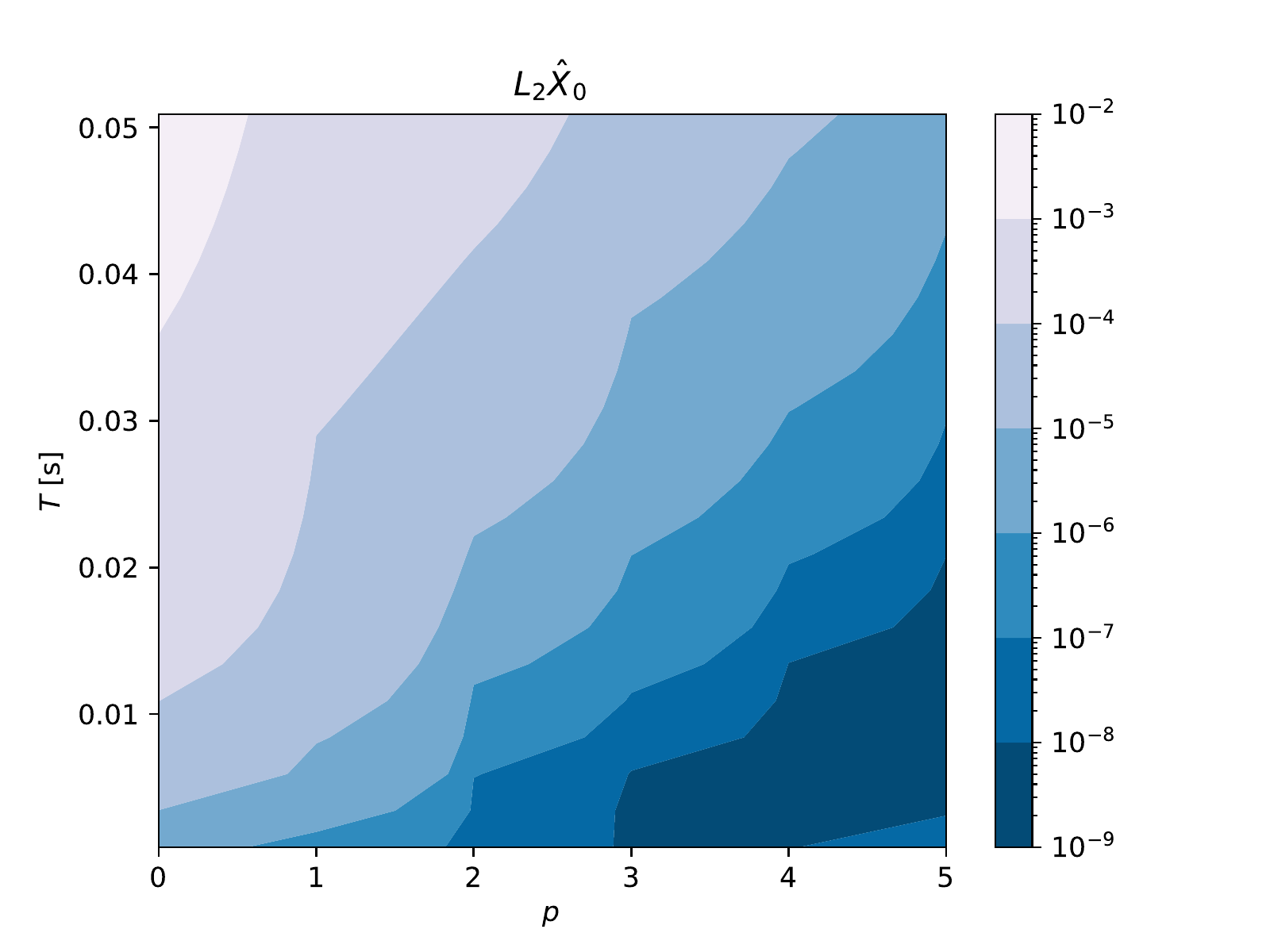} & 
    \hspace{-1.3cm} \includegraphics[scale=.42]{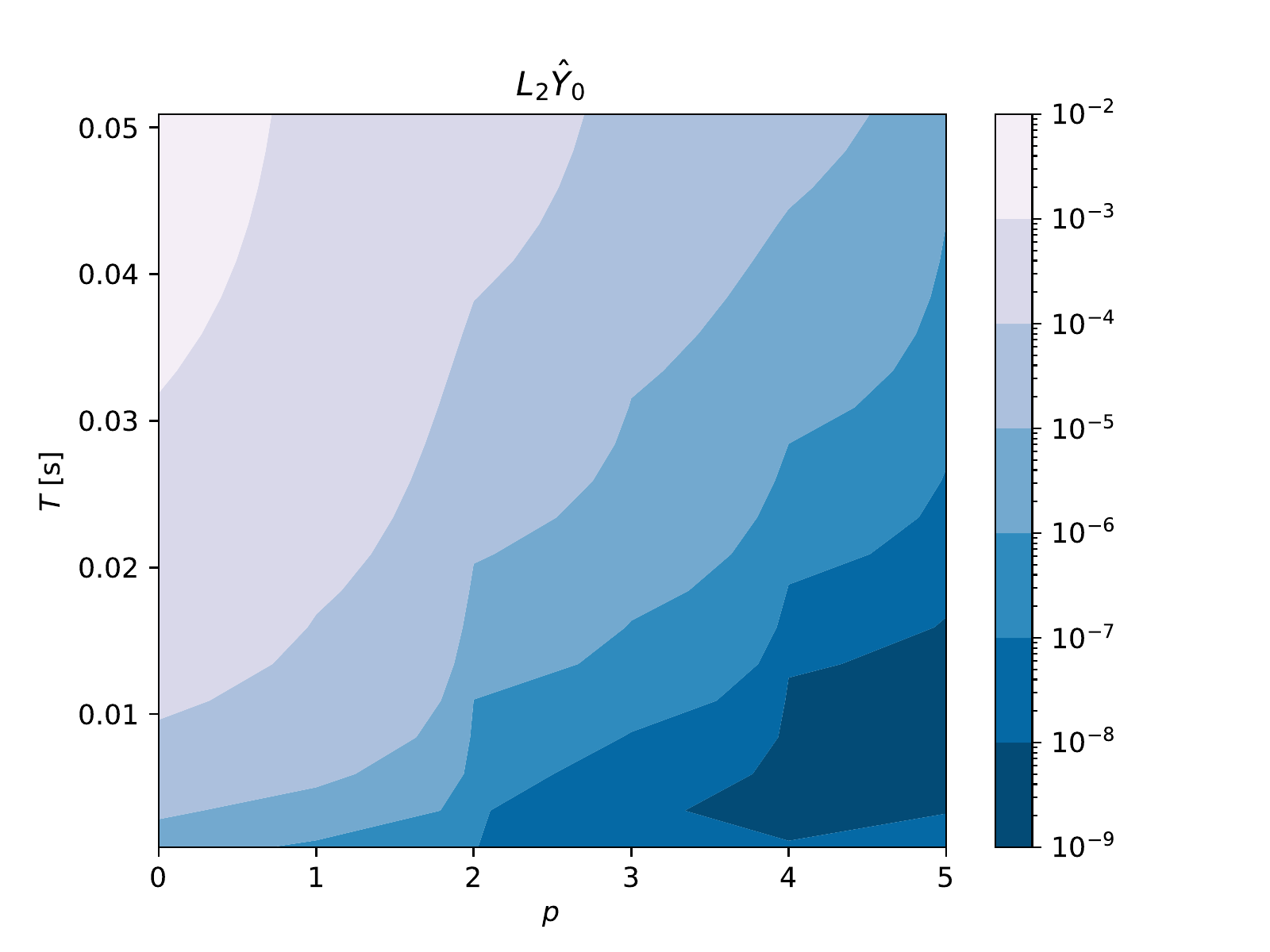} & 
    \hspace{-1.3cm} \includegraphics[scale=.42]{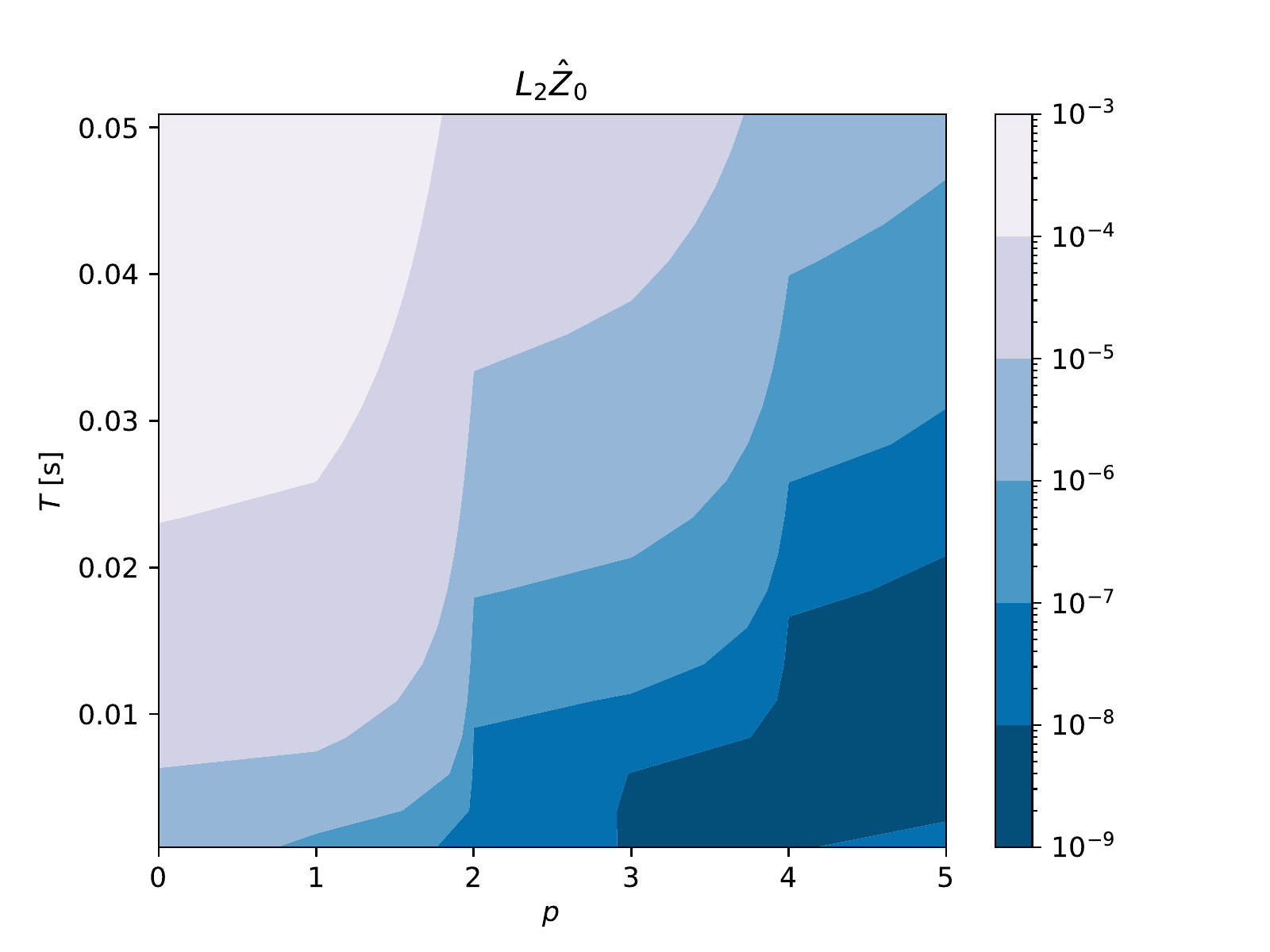}  
  \end{tabular}
 \caption{2D $L_2$-error plots $(p,T)$ for solutions $\V_0$ in the limit $T \rightarrow 0\,$s. Noting the logarithmic colour scale, we note that
     the error decays exponentially with $p$ for fixed $T$, and larger
     $T$ reduces the decay rate. 
  }
  \label{fig_2d_pT0}
  \end{figure}

 We will consider three regimes of averaging window size $T$, namely 
 (i) $T \rightarrow 0\,$s, (ii) $T \in [0.05, 0.5]\,$s, and (iii) $T \rightarrow \infty$,
 whilst studying the impact of higher order polynomial approximation in the solver 
 on the solutions' accuracy. 
 In practice, we achieve with our implementation and, e.g., $p = 5$ values for $T$ 
 of about $0.0005\,$s for case (i) and of $1e+12\,$s for case (iii), or even smaller respectively 
 larger $T$ for smaller $p$.
 We will consider an integration time of $tf = 167\,$s 
 here (corresponding to the first cycle) and we use the default adaptive timestepping error tolerance of $1.49012e^{-8}$ or otherwise stated.

 In Figure~\ref{fig_2d_pT0} we presented results for case (i) for small $T \in [0.001,0.0035,0.0060, ..., 0.05]\, $s. 
 The figure indicates that for $T \approx 0.005\,$s we have minimum error values 
 in $\hat X_0, \hat Y_0, \hat Z_0$ of about $10^{-9}$ for $p = 5$. 
 With these values we reach the default solver tolerance ($\approx 10^{-8}$), 
 while with a solver tolerance of about $10^{-10}$, we reach for the same $p$ and $T$ values
 minimum errors of about $10^{-11}$ that show a similar error pattern to Figure~\ref{fig_2d_pT0}.
 With even smaller tolerances the solutions become even more accurate (not shown).
 For each tolerance studied, we experience for values of $T$ smaller than about $0.0025\,$s 
 an increase in the error values (cf. smallest $T$ values in Figure~\ref{fig_2d_pT0}). 
 This increase is related to the fact that for $T \leq 0.0025\,$s the condition number of 
 the inverse of $M$ gets huge, leading to increased numerical errors in the 
 numerical solver of the ODE. These results confirm the discussion of Section~\ref{sec:asymptotics}.

 Next, we investigate case (ii). As shown in Figure~\ref{fig_2d_pTmiddle}, for averaging 
 window sizes smaller than $T\leq 0.2\,$s, the higher 
 order method significantly improves the accuracy (cf. also Figures~\ref{fig_Vxyz_1000s} 
 and \ref{fig_Vxyz_1000s_XYZzoom}). In particular 
 for $T = 0.05\,$s and $p =10$ we reach solver tolerance for all three 
 components, but also for other values with $T\leq 0.2\,$s, 
 our method improves the accuracy significantly. For values $0.2\,$s$\,\leq T\leq 0.5\,$s, the HO 
 method still provides improvements but with error values for $p=10$
 of $10^{-02}$ and $10^{-03}$ for $\hat X_0,\hat Y_0$ and $\hat Z_0$, respectively.
 For even larger averaging windows of $T\geq 0.5\,$s (case (iii)), 
 solutions share these latter error values even for larger values in 
 $p$ (not shown). 
 This is because in this case, the exponential function dominates the factors 
 in the expansion in \eqref{eqn_sum3} and, therefore, the results correspond to the 
 original averaging method of \citet{peddle2019parareal}.
 Also these results agree with the derivations of Section~\ref{sec:asymptotics}.

 \begin{figure}[h]\centering
 \begin{tabular}{ccc}
    \hspace{-1cm} \includegraphics[scale=.42]{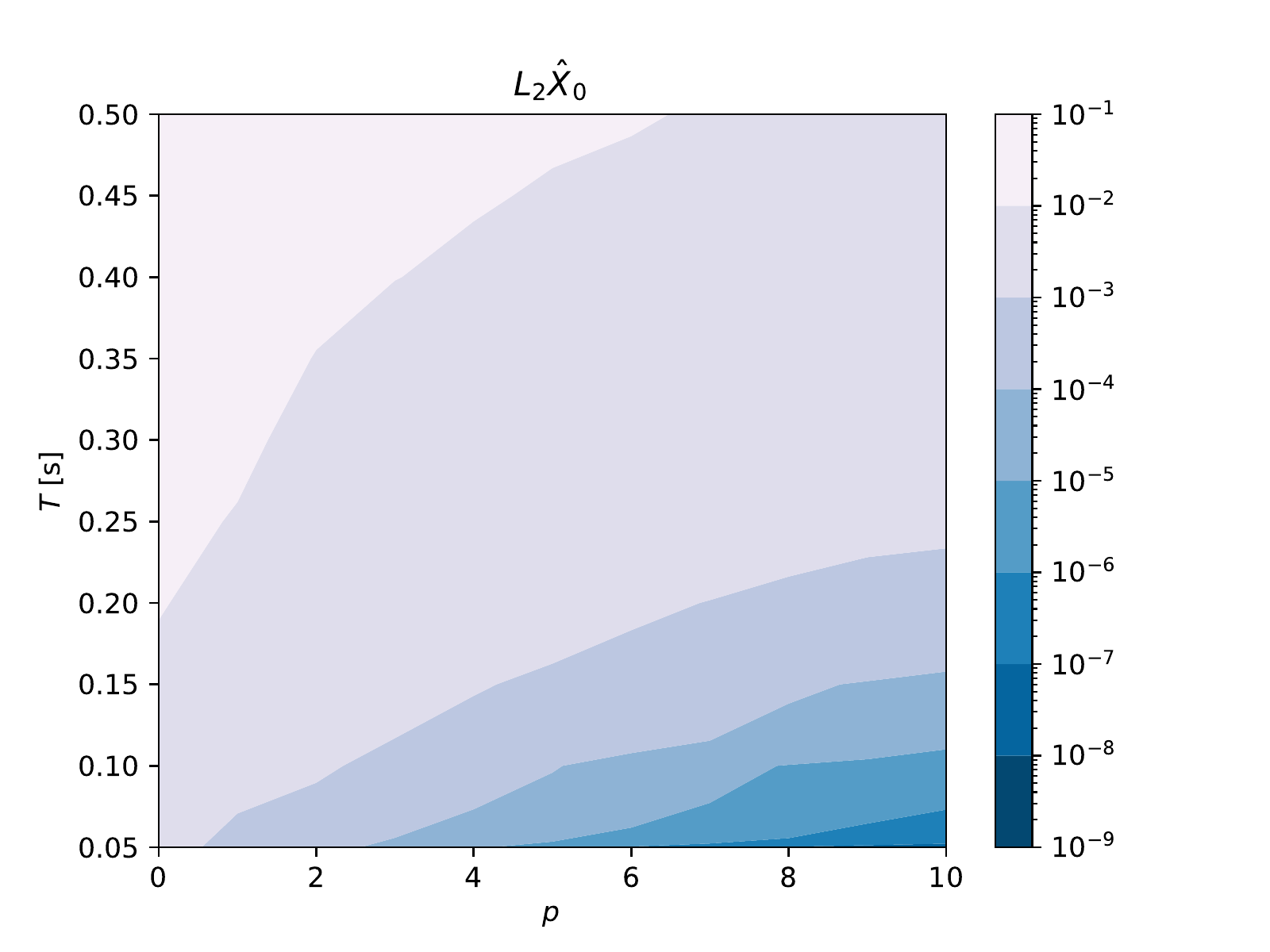} & 
    \hspace{-1.3cm} \includegraphics[scale=.42]{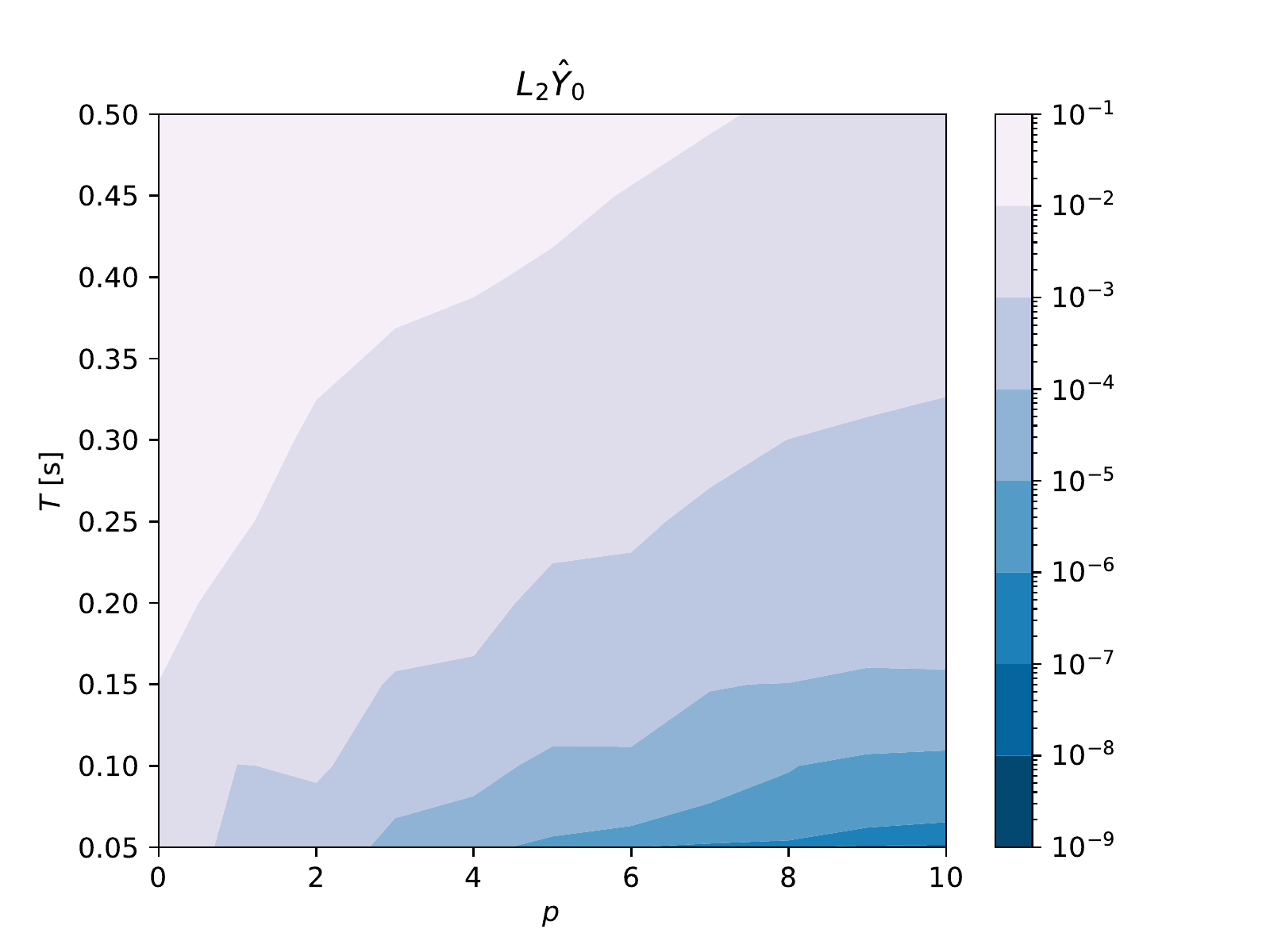} & 
    \hspace{-1.3cm} \includegraphics[scale=.42]{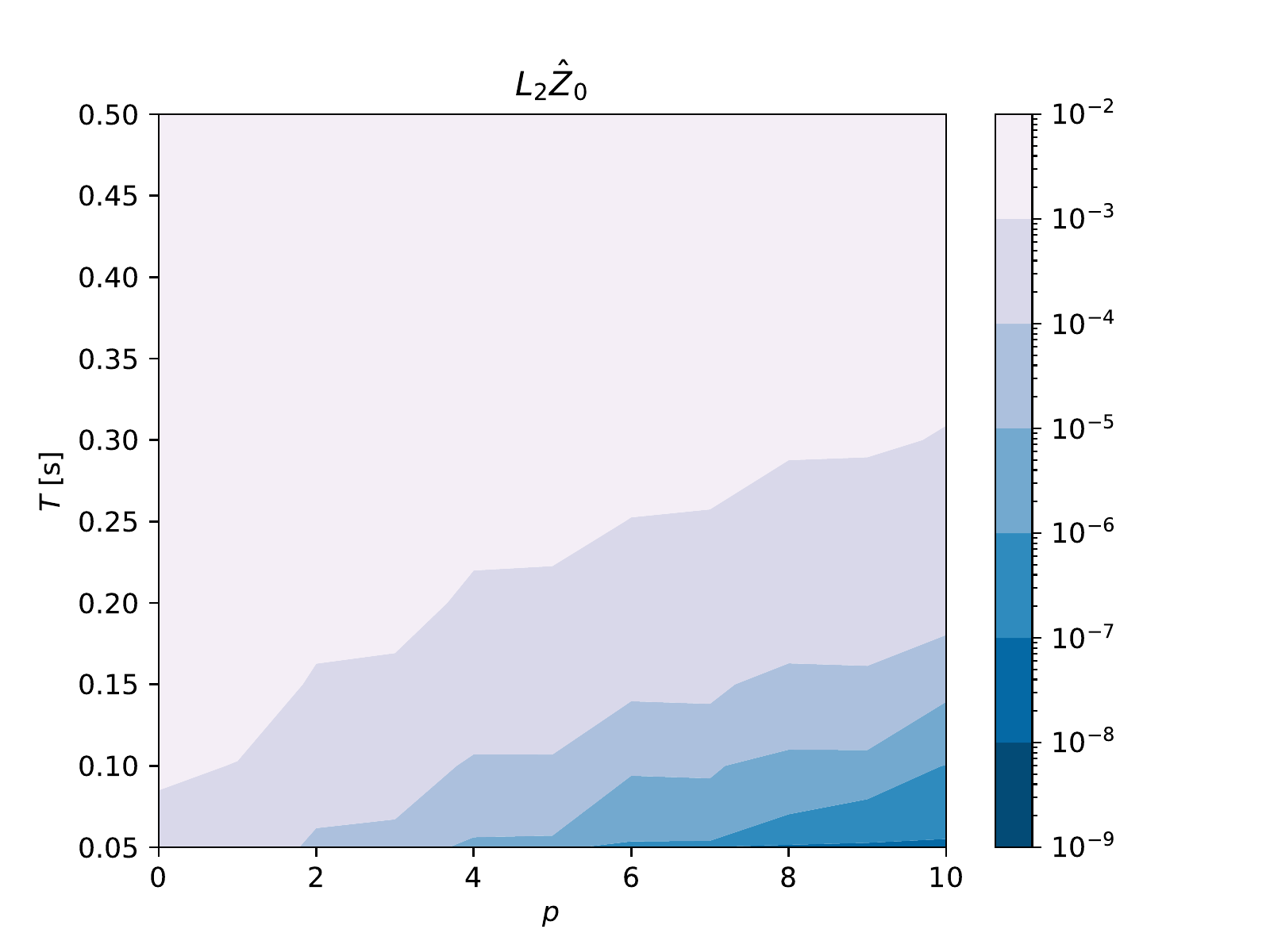}  
  \end{tabular}
  \caption{2D $L_2$-error plots $(p,T)$ for $\V_0$ for intermediate values of $T$.
  }
  \label{fig_2d_pTmiddle}
  \end{figure}

 \begin{figure}[h]\centering
 \begin{tabular}{cc}
    \includegraphics[scale=.45]{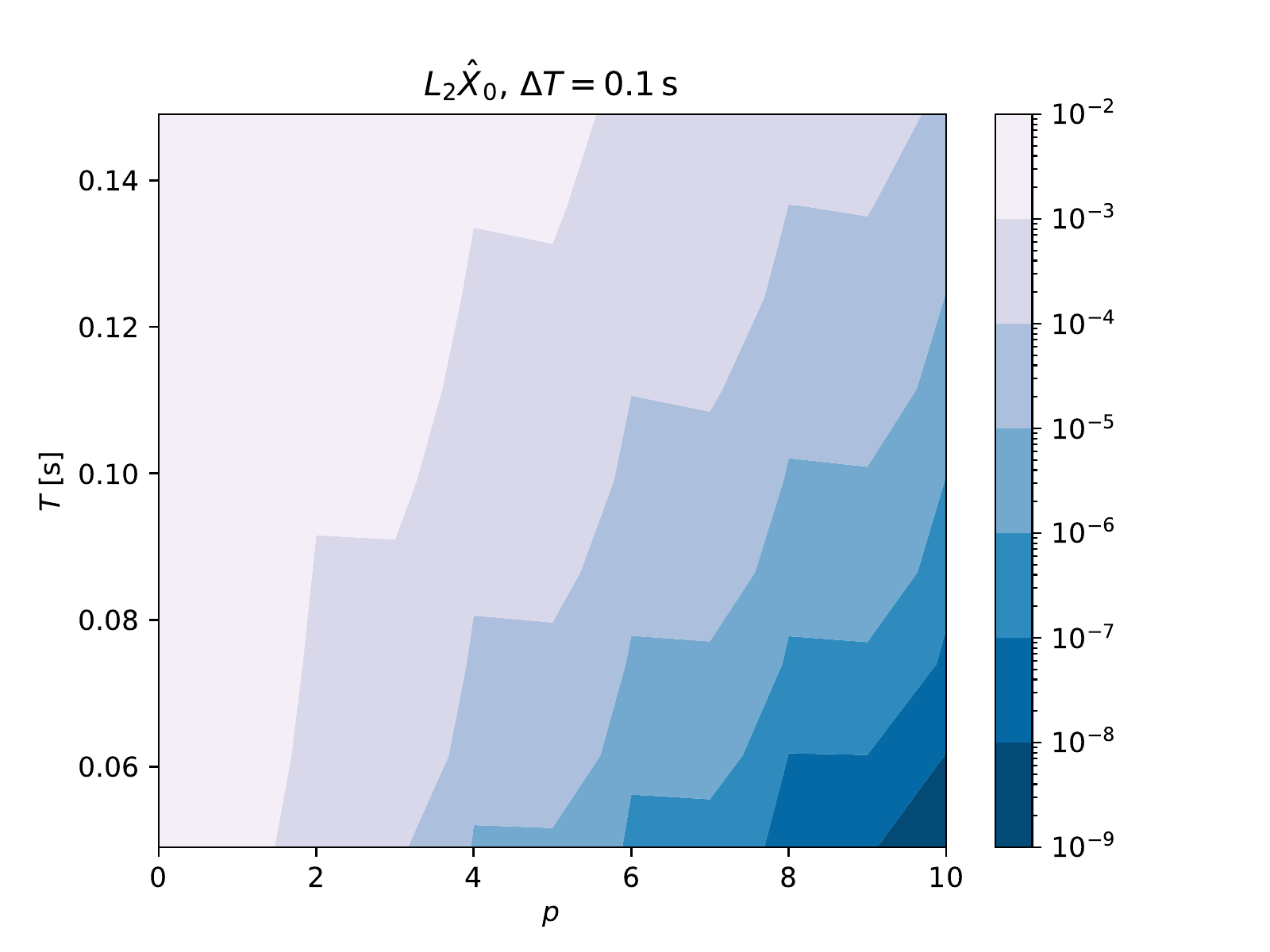}    &
    \includegraphics[scale=.45]{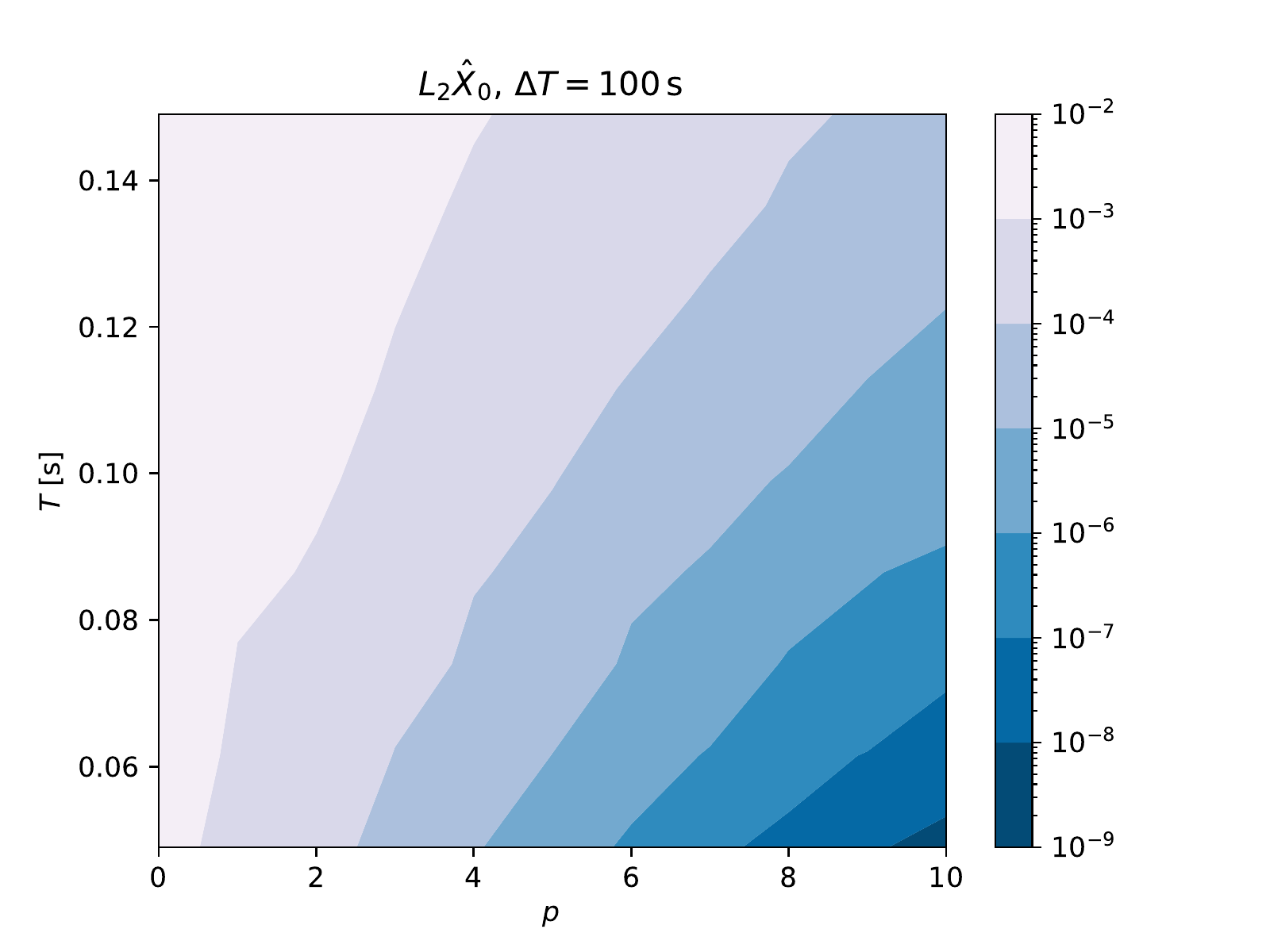}  
  \end{tabular}
  \caption{2D $L_2$-error plots $(p,T)$ for $\hat X_0$ for resetting windows 
  $\Delta T = 0.1\,s$ (left) and $\Delta T = 100\,s$ (right).
  }
  \label{fig_Vx1}
  \end{figure}
 
 As indicated above, the resetting after the time interval $\Delta T$ 
 of the higher order velocity components prevents the HO models from blowups.
 The latter are due to instabilities in the higher order terms 
 which, in turn, are caused by a growing discrepancy between principal and 
 higher order modes in $\V$ that might happen after long simulation times, 
 cf. Figure~\ref{fig_Vxyz_1000s}.  
 In Figure~\ref{fig_Vx1}, we present error plots for $\hat X_0$ solutions obtained
 with the higher order averaging approach for two different resetting values: 
 $\Delta T = 0.1\,$s (left) and $\Delta T = 100\,$s (right).
 The similarity of both error maps indicates that a change in the
 resetting window size $\Delta T$ has almost no impact on the accuracy of the solutions. 
 This property is shared by the corresponding error maps for the $\hat Y_0$ and $\hat Z_0$ 
 components (not shown).

\section{Summary and outlook}
\label{sec:summary}
In this paper we introduced a higher order finite window
averaging technique and investigated it by specialising to
Gaussian weight functions to allow for explicit computation of
averages. This enabled us to efficiently investigate the higher order
technique when applied to highly oscillatory ODEs, in particular
Lynch's swinging spring model. We found that higher order corrections
can strongly increase the accuracy of the averaged model solutions,
but only when the averaging window is close to the time period of the
fast frequency (this model only has two fast frequencies, in 2:1
resonance). We expect the situation to become more complicated when
there are a spectrum of fast frequencies. In this case, it is known
that very high frequencies do not change the slow components of the
solution much, but moderate fast frequencies can resonate in the
nonlinearity (just like in the swinging spring), altering the slow
dynamics. In this case one would want to select an averaging window
that removes as much fast dynamics as possible, whilst preserving accuracy
of the slow dynamics arising from near-resonant interactions, and the
higher order model could prove useful in reducing the impact of the
averaging on this slow dynamics.

In this paper we have also avoided numerical aspects, such as the
effects of approximate averages by numerical quadrature, and the
interaction of numerical time integrators with the
averaging.  \cite{haut2014asymptotic} investigated the impact
  of using numerical quadrature in the basic phase averaging setting
  (corresponding to the lowest order method in the present framework)
  and found that there are no numerical instabilities provided that
  the oscillations in the phase variable are sufficiently resolved by
  the quadrature (the Nyquist condition, essentially). In the present
  framework, choosing a different $\rho$, or using numerical
  quadrature, just results in a different inner product with which to
  compute the projections, so we do not anticipate difficulties.  An
analysis using the ideas of \citet{peddle2019parareal} would be useful
to address the interaction of numerical time integrators with the
averaging. Looking further ahead, if the higher order averaging model
is shown to improve the accuracy of averaging methods whilst still
allowing larger time steps in numerical integrations, this would
motivate the investigation of multilevel schemes such as PFASST, where
the low order averaging is used as a coarse propagator and higher
order averaging models provide higher order corrections, some of which
can be hidden in parallel behind first iterations in later timesteps.

\section*{Acknowledgement}

The authors would like to acknowledge funding from NERC NE/R008795/1 
and from EPSRC EP/R029628/1.

\appendix

\section{Fully explicit representation for $p = 2$}
\label{append_explicit}

 Here, we present a fully explicit version of \eqref{eqn_mass_2} in terms of tendencies
 for $\V_0$, $\V_1$, and $\V_2$. Split into these tendencies while inserting values for $\R_\alpha^m$
 and taking into account that $ \Fu_{m,j,k} = \Fu_{m,k,j} \forall m,j,k$,
 we arrive at 
 \begin{equation}\notag
 \begin{split}
  \dot\V_0(t) = & \sum_{m = 1}^M e^{ic_m t} e^{- \frac{c_m^2 T^2}{2} } 
\Big(  
\frac{3}{2}    \Fu_{m,0,0} 1 
+ \frac{3}{2}  \Fu_{m,1,1} (T^2 - k_m^2 T^4)  
+ \frac{3}{2}  \Fu_{m,2,2} (3T^4 - 6 k_m^2 T^6 +  k_m^4 T^8)   \\
&
+ 3\Fu_{m,0,1} (ic_m T^2)  
+ 3\Fu_{m,0,2} (T^2 - k_m^2 T^4)  
 + 3\Fu_{m,1,2} (i 3 k_m T^4 - i k_m^3 T^6)  \\
 & 
-\frac{1}{2} \Fu_{m,0,0} (1 - k_m^2 T^2) 
 -\frac{1}{2} \Fu_{m,1,1} (3T^2 - 6 k_m^2 T^4 +  k_m^4 T^6)  \\
 &
-\frac{1}{2} \Fu_{m,2,2} (15T^4  - 45 k_m^2 T^6    + 15 k_m^4 T^{8}  -  k_m^6 T^{10})  
-\Fu_{m,0,1} (i 3 k_m T^2 - i k_m^3 T^4)  \\
&
- \Fu_{m,0,2} (3T^2 - 6 k_m^2 T^4 +  k_m^4 T^6)  
-\Fu_{m,1,2} (i 15  k_m T^4 - i 10 k_m^3  T^6 +    i k_m^5 T^{8}) 
\big) ,
\end{split}
 \end{equation}
 which can be summarized to
\begin{equation}\notag
 \begin{split}
  \dot\V_0(t)= &\sum_{m = 1}^M e^{ic_m t} e^{- \frac{c_m^2 T^2}{2} } 
\Big(
   \Fu_{m,0,0} (1+ \frac{1}{2}c_m^2 T^2) 
 +  \Fu_{m,1,1} (\frac{3}{2}c_m^2 T^4 - \frac{1}{2} c_m^4 T^6) \\
 & +  \Fu_{m,2,2} (-3T^4  + \frac{27}{2} c_m^2 T^6 -6c_m^4 T^8 + \frac{1}{2} c_m^6 T^8) \\
 & +  \Fu_{m,0,1} (i c_m^3T^4)
 +    \Fu_{m,0,2} (3 c_m^2 T^4 - c_m^4 T^6)
 +    \Fu_{m,1,2} (- 6 i c_m T^4 +7 i c_m^3 T^6 - ic_m^5 T^8)
 \Big) .
 \end{split}
 \end{equation}
For $\V_1$, there follows directly
\begin{equation}\notag
 \begin{split}
  \dot\V_1(t) = & \sum_{m = 1}^M e^{ic_m t} e^{- \frac{c_m^2 T^2}{2} } 
\Big(\Fu_{m,0,0} (ic_m )  
+ \Fu_{m,1,1} (i 3 c_m T^2 - i c_m^3 T^4)  
+ \Fu_{m,2,2} (i 15  c_m T^4 - i 10 c_m^3  T^6 +    i c_m^5 T^{8})   \\
& + 2 \Fu_{m,0,1} (1 - c_m^2 T^2)   
+ 2 \Fu_{m,0,2} (i 3 c_m T^2 - i c_m^3 T^4)   
+ 2 \Fu_{m,1,2} (3T^2 - 6 c_m^2 T^4 +  c_m^4 T^6)   
  \Big) .
 \end{split}
 \end{equation}
 Finally for $\V_2$, we find
 \begin{equation}\notag
 \begin{split}
  \dot\V_2(t) = & \sum_{m = 1}^M e^{ic_m t} e^{- \frac{c_m^2 T^2}{2} } 
\Big(-\frac{1}{2}\Fu_{m,0,0} 1/T^2 
-\frac{1}{2}\Fu_{m,1,1} (T^0 - k_m^2 T^2) 
-\frac{1}{2}\Fu_{m,2,2} (3T^2 - 6 k_m^2 T^4 +  k_m^4 T^6)\\
& 
- \Fu_{m,0,1} (ik_m T^0)  
- \Fu_{m,0,2} (T^0 - k_m^2 T^2) 
- \Fu_{m,1,2} (i 3 k_m T^2 - i k_m^3 T^4) \\
&
+ \frac{1}{2}\Fu_{m,0,0} (1/T^2 - k_m^2 T^0) 
+\frac{1}{2} \Fu_{m,1,1} (3T^0 - 6 k_m^2 T^2 +  k_m^4 T^4) \\
&
+\frac{1}{2} \Fu_{m,2,2} (15T^2  - 45 k_m^2 T^4    + 15 k_m^4 T^{6}  -  k_m^6 T^{8}) \\
&
+ \Fu_{m,0,1} (i 3 k_m T^0 - i k_m^3 T^2) 
+ \Fu_{m,0,2} (3T^0 - 6 k_m^2 T^2 +  k_m^4 T^4) \\
&
+ \Fu_{m,1,2} (i 15  k_m T^2 - i 10 k_m^3  T^4 +    i k_m^5 T^{6})  
\Big),  
\end{split}
\end{equation}
which leads to
\begin{equation}\notag
 \begin{split}
  \dot\V_2(t) =
&\sum_{m = 1}^M e^{ic_m t} e^{- \frac{c_m^2 T^2}{2} } 
\Big(  \Fu_{m,0,0}  (-\frac{1}{2}c_m^2)
 +  \Fu_{m,1,1} (1 - \frac{5}{2}c_m^2 T^2 +\frac{1}{2}c_m^4 T^4) \\
 &
  +  \Fu_{m,2,2} (6 T^2 - \frac{39}{2} c_m^2 T^4  + 7 c_m^4 T^6 - \frac{1}{2} c_m^6 T^8) \\
 & +  \Fu_{m,0,1}  (2ic_m)
 +    \Fu_{m,0,2}  (2 -5 c_m^2 T^2 + c_m^4 T^4)
 +    \Fu_{m,1,2}  (12 i c_m T^2 - 9 i c_m ^3 T^4  + i c_m^5 T^6)
\Big).
 \end{split}
 \end{equation}
 

\end{document}